\documentclass[11pt]{amsart}
	\usepackage{amssymb, latexsym, amsmath, amsfonts}
	\usepackage{wrapfig}
	\usepackage{subcaption}
	\usepackage[pdftex]{graphicx}
	\usepackage{float}
	\usepackage[numbers]{natbib}
	\usepackage{tikz}
	\usepackage{hyperref}
	\newtheorem{thm}{Theorem}[section]
	\newtheorem{cor}[thm]{Corollary}
	\newtheorem{lem}[thm]{Lemma}
	\newtheorem{prop}[thm]{Proposition}
	\theoremstyle{definition}
	\newtheorem{defn}[thm]{Definition}

	\theoremstyle{remark}
	\newtheorem{rem}[thm]{Remark}
	\numberwithin{equation}{section}
	
	\numberwithin{equation}{section}

	\newcommand{\mbb}{\mathbb}

	\newcommand{\ov}{\overline}
	
	\newcommand{\ep}{\epsilon}
	\newcommand{\no}{\noindent}

	\newcommand{\cal}{\mathcal}
	\newcommand{\la}{\lambda}

\begin{document}
	\title{Examples of non--autonomous basins of attraction }
	%\thanks{}
	
	\author{Sayani Bera, Ratna Pal and Kaushal Verma} 
\address{Sayani Bera: Mathematics, Harish-Chandra Research Institute, HBNI, Allahabad-211019, India}
	\email{sayanibera@hri.res.in}
\address{Ratna Pal: Department of Mathematics, Indian Institute of Science, Bangalore 560012, India}
\address{
Current Address: Department of Mathematics, Indian Institute of Science Education and Research, Pune, Maharashtra-411008, India}
	\email{ratna@iiserpune.ac.in}
	\address{Kaushal Verma: Department of Mathematics, Indian Institute of Science, Bangalore 560012, India}
    \email{kverma@math.iisc.ernet.in}
	\pagestyle{plain}
	
	\begin{abstract}
The purpose of this paper is to present several examples of non--autonomous basins of attraction that arise from sequences of automorphisms of $\mathbb C^k$. In the first part, we prove that the non--autonomous basin of attraction arising from a pair of automorphisms of $\mathbb C^2$ of a prescribed form is biholomorphic to $\mathbb C^2$. This, in particular, provides a partial answer to a question raised  in \cite{Survey} in connection with Bedford's Conjecture about uniformizing stable manifolds. In the second part, we describe three examples of Short $\mathbb C^k$'s with specified properties. First, we show that for $k \geq 3$, there exist $(k-1)$ mutually disjoint Short $\mathbb C^k$'s in $\mathbb C^k$. Second, we construct a Short $\mathbb C^k$, large enough to accommodate a Fatou--Bieberbach domain, that avoids a given algebraic variety of codimension $2$. Lastly, we discuss examples of Short $\mathbb C^k$'s with (piece--wise) smooth boundaries.
\end{abstract}
	\maketitle

%%%%%%%%%%%%%%%%%%%%%%%%%%%%%%
	\section{Introduction}
	
\no 	
%\cite{Peters-Smit}, \cite{Abate-et-al}\cite{Fornaess}, \cite{Jonsson-Varolin}
Let $f$ be a holomorphic automorphism of a complex manifold $X$ equipped with a Riemannian metric, say $d_X$. Suppose $K \subset X$ is an invariant compact set on which $f$ is uniformly hyperbolic. For $p \in X$, let $\Sigma^s_f(p)$ be the stable manifold of $f$ through $p$, i.e.,
\[
\Sigma^s_f(p) = \{ z \in X : d_X(f^{\circ n}(z), f^{\circ n}(p)) \rightarrow 0 \text{ as } n \rightarrow \infty \}.
\]
By the Stable manifold theorem, $\Sigma^s_f(p) \subset X$ is an immersed complex submanifold, say of dimension $k$ and this turns out to be diffeomorphic to $\mathbb R^{2k}$. The question of whether $\Sigma^s_f(p)$ is biholomorphic to $\mathbb C^k$ for every $p \in K$ was raised by Bedford in \cite{openproblems}. While this is known to be true in several cases (see for example \cite{Jonsson-Varolin}, \cite{Abate-et-al} and \cite{Peters-Smit}), a result of Forn\ae ss--Stens\o nes in \cite{FSt} shows that $\Sigma^s_p(f)$ is biholomorphic to a domain in $\mathbb C^k$ for {\it every} $p \in K$. This was done by studying a related question which, more importantly, is a reformulation of Bedford's question:

\medskip

\no {\it Conjecture:} Let $\{ F_n \}$ be a sequence of automorphisms of $\mathbb C^k$ satisfying
\begin{equation}
 C \Vert z \Vert \le \Vert F_n(z) \Vert \le D \Vert z \Vert
\end{equation}
for $z \in B^k(0; 1)$ (the unit ball around the origin) and $0 < C, D < 1$. Then the basin of attraction of $\{F_n\}$ at the origin defined as
\[
\Omega_{\{F_n\}} = \{ z \in \mathbb C^k : F_n \circ F_{n-1} \circ  \cdots \circ F_1(z) \rightarrow 0 \text{ as } n \rightarrow \infty \} 
\]
is biholomorphic to $\mathbb C^k$.

\medskip

\no On the other hand, the necessity of having such uniform bounds for each $F_n$ on the unit ball was shown by Forn\ae ss in \cite{ShortC2}. In particular, if $\{H_n\}$ is a sequence of automorphisms of $\mathbb C^2$ of the form
\[
H_n(z, w) = (a_n w + z^2, a_n z)
\]
where $0 < \vert a_{n+1} \vert \le \vert a_n \vert^2$ and $0 < \vert a_0 \vert < 1$, then the corresponding basin $\Omega_{\{H_n\}}$ is {\it not} biholomorphic to $\mathbb C^2$ since it was shown to admit a non-constant bounded plurisubharmonic function. Note that the $H_n$'s do not satisfy a uniform bound condition near the origin. In this case, $\Omega_{\{H_n\}}$ can be written as the limit of an increasing union of domains each of which is biholomorphic to the unit ball in $\mathbb C^2$. Furthermore, the infinitesimal Kobayashi metric on $\Omega_{\{H_n\}}$ vanishes identically. Thus, $\Omega_{\{H_n\}}$ is neither all of $\mathbb C^2$ nor a Fatou--Bieberbach domain. Such a domain was christened {\it Short} $\mathbb C^2$ (or more generally, a {\it Short} $\mathbb C^k$ if the domain sits in $\mathbb C^k$, $k \geq 3$) in \cite{ShortC2}.

\medskip

\no As explained in \cite{ShortC2}, the existence of such domains is intrinsically linked with a version of the Levi problem namely, to decide whether the union of an increasing sequence of Stein domains is Stein. A counterexample constructed by Fornaess \cite{Fornaess} shows that this is not true in general if $k \geq 3$. However, in $\mathbb C^2$, Fornaess--Sibony \cite{FoSi} were able to classify those domains which arise as the increasing union of biholomorphic images of the ball and which additionally satisfy the property that the Kobayashi metric does not vanish identically. The other possibility is when the Kobayashi metric vanishes identically -- and this is where {\it Short} $\mathbb C^2$'s make their appearance.

\medskip

\no The purpose of this paper is two fold. First, we will study a seemingly straightforward version of the conjecture mentioned above that was stated as Problem $22$ in \cite{Survey}. We recall the statement below:

\medskip

\no {\it Problem:} Let $F$ and $G$ be automorphisms of $\mathbb C^k$ both having an attracting fixed point at the origin. Let $\{f_n\}$ be a sequence in which each $f_n$ is either $F$ or $G$. Is  the basin of attraction $\Omega_{\{f_n\}}$ biholomorphic to $\mathbb C^k$?

\medskip

\no Furthermore in \cite{Survey}, the maps defined by
\[
F(z_1, z_2) = (\alpha z_1 + z_2^2, \beta z_2)  \; \text{and} \; G(z_1, z_2) = (\beta z_1, \alpha z_2 + z_1^2)
\]
where $\vert \alpha \vert, \vert \beta \vert > 0$ and $\vert \alpha \vert^3 > \vert \beta \vert$ (for example, let $\alpha = 1/2$ and $\beta = 1/9$) were proposed as test cases to study. A moment's thought shows that this problem reduces to studying the non--autonomous basin of attraction of the sequence $\{H_{p(k), q(k)} : k \ge 1\}$ where $p, q : \mathbb N \rightarrow \mathbb N \cup \{0\}$ and
\[
H_{p(k), q(k)} = F^{p(k)} \circ G^{q(k)}.
\]
This observation is used to prove the following result:
\begin{thm}\label{theorem 5}
Let $F$ and $G$ be automorphisms of $\mbb{C}^2$ with an attracting fixed point at the origin such that 
 the matrices $A=F'(0)$ and $B=G'(0)$ are as follows:
        \[
A=
  \begin{pmatrix}
    \alpha & 0 \\
    0 & \beta
  \end{pmatrix} \; \text{and}\;
 B=
 \begin{pmatrix}
    \beta & 0 \\
    0 & \alpha
  \end{pmatrix}\; 
\]
with $|\alpha|^r < |\beta|$ for some $r>2.$ Then the basin of attraction of $\{H_{p(k), q(k)} : k \ge 1\}$ at the origin is biholomorphic to $\mbb C^2$ in the following cases:
\begin{enumerate}
\item[(i)] For every $k \ge 1$, $q(k)$ is bounded, $p(k) \ge  1$ and $(2 + \frac{3}{r-2})p(k) - q(k) \geq 0.$
\item[(ii)] For every $k \ge 1$,  $p(k)$ is bounded, $q(k) \ge 1$ and $(2 + \frac{3}{r-2})q(k) - p(k) \geq 0$.
\end{enumerate}
\end{thm}

\no Moreover, we will use methods similar to those in \cite{peters} to obtain a related result for more specific automorphisms of $\mathbb C^2$.

\begin{prop}\label{two maps}
Consider the automorphisms $$F(z_1,z_2)=(\alpha z_1+z_2^2,\beta z_2)\text{ and }G(z_1,z_2)=(\beta z_1, \alpha z_2+z_1^k)$$ where $|\alpha|^k<|\beta| \le |\alpha|^{k-1}$ for $k\geq 2$. Let $\{F_n\}$ be a sequence in which each $F_n$ is either $F$ or $G$. Then the non--autonomous basin of attraction at the origin of $\{F_n\}$ is biholomorphic to $\mathbb C^2$.
\end{prop}

\medskip\no 
Second, let us recall a classical result of Rosay--Rudin \cite{RR1} which shows that if an automorphism of $\mathbb C^k$ has an attracting fixed point, then the associated basin of attraction is a Fatou--Bieberbach domain. This result forms the basis of several examples of Fatou--Bieberbach domains with prescribed properties that were constructed by them in \cite{RR1}. In the same vein, it is natural to ask whether it is possible to construct {\it Short} $\mathbb C^k$'s with specified properties. In what follows, we provide several examples of {\it Short} $\mathbb C^k$'s that satisfy additional properties -- these being much in the spirit of what is known for Fatou--Bieberbach domains.

\medskip\no 
For the first example, recall that a shift--like map of type $\nu$ (where $1 \le \nu \le k-1$) is an automorphism of $\mathbb C^k$ given by
\[
S(z_1,z_2, \hdots,z_k)=(z_2,z_3,\hdots,z_k, h(z_{k-\nu+1})-az_1)
\]
where $h$ is an entire function on $\mathbb C$ and $a \in \mathbb C^{\ast}$. These maps were introduced and studied by Bedford--Pambuccian \cite{BP}. By working with suitable shift--like maps and the filtration they preserve, it is possible to create a mutually disjoint finite collection of {\it Short} $\mathbb C^k$'s.

\begin{thm}\label{theorem 2}
For $k \ge 3$, there exist $(k-1)$ {\it Short} $\mathbb C^k$'s, say $\Omega_i$, $1 \le i \le k-1$, such that
\[
\bigcup_{i = 1}^{k-1} \Omega_i \subset \mathbb C^k \text{ and } \Omega_i \cap \Omega_j = \emptyset \text{ for } i \not= j.
\]
\end{thm}

\no 
For the second example, recall that Hubbard--Buzzard \cite{HuBu} have constructed a Fatou--Bieberbach domain that avoids a given algebraic variety of codimension $2$ in $\mathbb C^k$ for $k \ge 3$. Again, by working with suitable shift--like maps and by controlling the rate of convergence of their linear parts, we have:

\begin{thm}\label{theorem 1}
Let $V \subset \mathbb C^k$, $k \geq 3$, be an algebraic variety of codimension $2$. Then there exists a Short $\mathbb C^k$, say $\Omega$ and a Fatou--Bieberbach domain, say $D$ such that $D \subset \Omega \subset \mathbb C^k \setminus V$.
\end{thm}
	
\no Let us recall the methods of constructing a {\it Short} $\mathbb C^2$ from \cite{ShortC2}:
\begin{itemize}
\item Let $\{F_n\}$ be a sequence of automorphisms of $\mbb{C}^2$ of the form 
\[ F_n(z_1,z_2)=(a_n z_2+z_1^2, a_n z_1)\] where $0 < |a_{n+1}| \le |a_n|^2$ and $0 <|a_0|<1$. Let $\phi$ and $a_{\infty}$ be defined as
\[ \phi= \lim_{n \to \infty} \frac{\log \phi_n(z)}{d^n} \text{ and }a_{\infty}=\lim_{n \to \infty}a_n^{-2^n}\] where \[\phi_n=\max\{|\pi_i \circ F_n \circ \hdots F_0(z)|, |a_n|: 1 \le i \le 2\}.\]
Here, and in what follows $\pi_i$ is the standard projection map on the $i-$th coordinate for $1 \le i \le k$. 
\medskip\no 
Then for every $c > \log a_{\infty}$, the $c$--sublevel set of $\phi$, i.e.,
\[
\{z \in \mathbb C^2: \phi < c\}
\]
is a {\it Short} $\mathbb C^2$. Moreover, the $0-$sublevel set of $\phi$ is the non--autonomous basin of attraction for the sequence $\{F_n\}$.

\medskip

\item Let $H$ be a  H\'{e}non map of the form
\[ H(z_1,z_2)=(z_2,\delta z_1+z_2^d).\] Then for every $c >0$, the $c$--sublevel set of the positive Green's function of $H$, i.e.,
\[ G^+(z)=\lim_{n \to \infty} \frac{1}{d^n}\log^+\|H^n(z)\|\]
 is a {\it Short} $\mathbb C^2$.

\end{itemize}
Both $\phi$ and $G^+$ are known to be pluriharmonic on the sets 
\[\{z \in \mbb C^2: \phi(z)>\log a_{\infty} \} \text{ and } \{ z \in \mbb C^2: G^+(z)>0\}\] respectively. Hence, by Sard's theorem, for most values of $c$ in the admissible range, the $c$--sublevel sets of either $\phi$ or $G^+$ are {\it Short} $\mathbb C^2$'s with smooth boundary. However, this does not ensure that the $0$--sublevel set of $\phi$, which is the non--autonomous basin of attraction at the origin, always has smooth boundary. In this regard, we have:
\begin{thm}\label{theorem 3}
There exists a {\it Short} $\mathbb C^2$ with $C^{\infty}$--smooth boundary that arises as the non-autonomous basin of attraction for a sequence of automorphisms having an attracting fixed point at the origin.
\end{thm}

\no This is motivated by Stens\o nes's proof of the existence of Fatou--Bieberbach domains with $C^{\infty}$--smooth boundary (see \cite{St}). We follow a similar approach for {\it Short} $\mathbb C^2$'s as well, namely we try to control the behaviour of the boundary on a large enough polydisc and then exhaust all of $\mathbb C^2$ with polydiscs of increasing size. A construction of a {\it Short $\mathbb{C}^2$} with $C^{\infty}$--smooth boundary on a fixed polydisc was also given by Console in \cite{console}.

\medskip\no 
By adopting some methods from Globevnik's work \cite{Gl} and \cite{console}, it is possible to retain boundary smoothness and at the same time avoid a given variety of codimension $2$.

\begin{thm}\label{theorem 4}
Let $V \subset \mathbb C^k$, $k \ge 3$, be an algebraic variety of codimension $2$. Then there exists a sequence of automorphisms $\{F_n\}$ of $\mathbb C^k$ with a common attracting fixed point, say $p$ such that the basin of attraction $\Omega_{\{F_n\}}$ at $p$ is a piecewise smooth {\it Short} $\mathbb C^k$ that does not intersect $V$.
\end{thm}

%%%%%%%%%%%%%%%%%%%%%%%%%%%
\no {\it Acknowledgements:} The authors would like to thank Greg Buzzard, Josip Globevnik and Han Peters for their encouragement and very helpful pointers.

%%%%%%%%%%%%%%%%%%%%%%%%%%%%%%%%%%%%%%%%%%%%
\section{Proof of Theorem \ref{theorem 5} and Proposition \ref{two maps}}

\no Let us recall the following from \cite{peters}. 

\medskip

\no \begin{defn}
Let $F$ be an automorphism of $\mbb C^k $, $k \ge 2$ having a fixed point at the origin such that $F'(0)$ is a lower triangular matrix. We say $F$ is {\it correctly ordered} if the diagonal entries of $F'(0)$, i.e., $\la_1,\hdots, \la_k$ (from upper left to lower right) satisfy the condition
\[ |\la_j ||\la_i| < |\la_l|,\]
for $1 \le l \le j \le k$ and any $1 \le i \le k.$
\end{defn}

\begin{defn}
A family $\cal{F}$ of {\it correctly ordered} automorphisms of $\mbb C^k$, $k \ge 2$ is said to be {\it uniformly attracting} if there exist $0< C< D< 1$ such that for every $z \in B^k(0;1)$ and $F \in \cal{F}$
\[ C\|z\| \le \|F(z)\| \le D\|z\|\] and there exists $0 < \xi < 1$ such that 
\begin{align}\label{eigen condition}
 |\la_j ||\la_i| < \xi |\la_l|
\end{align}
for $1 \le l \le j \le k$ and any $1 \le i \le k.$
\end{defn}

\begin{thm}[Peters, H]\label{theorem 9}
Let $\{F_n\}$ be a uniformly attracting sequence of automorphisms of $\mbb{C}^k$, $k \ge 2.$ Then the basin of attraction of $\{F_n\}$ at the origin is biholomorphic to $\mbb C^k.$
\end{thm}

\no We will use the following version of this theorem which is valid when $k=2$.

\medskip

\no {\it Let $\{F_n\}$ be a uniformly attracting sequence of automorphisms of $\mbb{C}^2$  and let $\{\lambda_{1,n}\}$, $\{\lambda_{2,n}\}$ be the diagonal entries of $\{F_n'(0)\}$. Since $F_n'(0)$ is a lower triangular matrix by assumption, $\lambda_{1,n}$ and $\lambda_{2,n}$ are the eigenvalues of $F_n'(0)$.  If there exists $0<\xi<1$ such that
 \[
|\lambda_{1,n}|< \xi, \; |\lambda_{2,n}|< \xi \; \text{and}  \;|\lambda_{2,n}|^2\le \xi |\lambda_{1,n}|,
\]
the basin of attraction of the sequence $\{F_n\}$ at the origin is biholomorphic to $\mbb C^2$.}

\medskip

\no However, by assumption $\{F_n\}$ is a uniformly attracting sequence of automorphisms and hence there exists $0<\xi_1<1$ such that 
\[
|\lambda_{1,n}|< \xi_1, \; |\lambda_{2,n}|< \xi_1
\]
for each $n \ge 1$. Therefore (\ref{eigen condition}) reduces to
\begin{align}\label{condition}
|\lambda_{2,n}|^2\le \xi |\lambda_{1,n}|
\end{align} 
for some $\xi$, where $0< \xi_1< \xi< 1$.

\medskip

\no We will complete the proof of Theorem \ref{theorem 5} by showing the existence of a suitable $\xi$ so that (\ref{condition}) holds.

\begin{proof}[Proof of Theorem \ref{theorem 5}]

It can be seen that   
\[
DF^{p(k)}(0)= \begin{pmatrix}
    \alpha^{p(k)} & 0 \\
    0 & \beta^{p(k)}
  \end{pmatrix} \text{ and }
DG^{q(k)}(0)=\begin{pmatrix}
    \beta^{q(k)} & 0 \\
    0 & \alpha^{q(k)}
  \end{pmatrix} 
\]
and hence  
\[ DH_{p(k),q(k)}(0)=\begin{pmatrix}
    \alpha^{p(k)} \beta^{q(k)} & 0 \\
    0 & \alpha^{q(k)} \beta^{p(k)}
    \end{pmatrix}
    \]
    for all $p(k), q(k) \ge 1.$
    
 \medskip 

\no It is sufficient to find a uniform $0 < \xi < 1$ so that
\[
\vert \alpha \vert^{2q(k)} \vert \beta \vert^{2p(k)} < \xi \vert \alpha \vert^{p(k)} \vert \beta \vert^{q(k)}
\]
for all $k \ge 1$. First, some reductions are in order. If $q(k) = 0$ for $k \ge k_0 \ge 0$, then $F_n=F$ for all large enough $n$. In this case, Rosay--Rudin \cite{RR1} show that the basin of attraction of $\{F_n\}$ is all of $\mbb C^2.$  We can therefore assume that $q(k) \ge 1$ for infinitely many $k$'s. The same reasoning applies to $p(k)$ as well. Further, by the assumptions in Theorem 1.1 (i), there exists a uniform $M$ such that $0 \le q(k) \le M$ for all $k$. There are two possibilities for $p(k)$, namely $p(k) \le M$ or $p(k) > M$ depending on $k$. If $p(k) \le M$, we leave $H_{p(k), q(k)}$ undisturbed. Else, if $p(k) > M$, let $p(k) = MN(k) + R(k)$ where $0 \le R(k) < M$. Observe that
\begin{align*}
H_{p(k), q(k)} &= F^{R(k)} \circ \underbrace{F^M \circ  \cdots \circ F^M}_{N(k)} \circ H_{M, q(k)}\\
                        &=H_{R(k), 0} \circ \underbrace{H_{M, 0} \circ \cdots \circ H_{M, 0}}_{N(k)} \circ H_{M, q(k)}.
\end{align*}
In other words, $H_{p(k), q(k)}$ is the composition of maps, all of which are of the form $H_{\gamma(k), \delta(k)}$ with $\gamma(k) \le M$ and $\delta(k) \le M$. Therefore, it is possible to create a new sequence, still denoted by $\{H_{p(k), q(k)}\}$, from the given one so that 
\[
0 \le p(k) \le M \; \text{and} \; 0 \le q(k) \le M.
\]

\no We may assume, by rearranging the sequence $\{ H_{p(k), q(k)} \}$ if needed, that both $p(k)$ and $q(k)$ are maximal in their ranges for every $k \ge 1$. The boundedness of both $p(k)$ and $q(k)$ ensures that the sequence $\{H_{p(k), q(k)}\}$ has a uniform bound on the rate of contraction on a sufficiently small ball around the origin. 

\medskip

\no {\it Case 1:} Suppose that $2p(k) - q(k) \geq 0$. In this case, note that
\[
\vert \alpha \vert^{2q(k) - p(k)} \vert \beta \vert^{2p(k) - q(k)} \le \vert \alpha \vert^{p(k) + q(k)} \leq \vert \alpha \vert 
\]
where the last inequality holds since $p(k) \geq 1$. Here, in this case, we do not have to worry about how $\vert \alpha \vert$ and $\vert \beta \vert$ are related.

\medskip

\no {\it Case 2:} Suppose that $2p(k) - q(k) < 0$. In this case, let $0 < \eta < 1$ be such that
\[
\vert \alpha \vert^r \le \eta \vert \beta \vert < \vert \alpha \vert.
\]
Note that $2q(k) - p(k) > p(k) + q(k) \geq 1$. Then
\begin{align*}
\vert \alpha \vert^{2q(k)-p(k)} \vert \beta \vert^{2p(k)-q(k)} & \leq  \left(\eta \vert \beta \vert \right)^{(2q(k)-p(k))/r} \vert \beta \vert^{2p(k) - q(k)}\\
                                                                                                    & \leq  \eta^{(2q(k)-p(k))/r} \vert \beta \vert^{\big( (2q(k)-p(k))/r + 2p(k)-q(k) \big) }.
\end{align*}
Simplifying the exponent of $\vert \beta \vert$ and noting that $2q(k) - p(k) \geq 1$, the last term above is dominated by
\[
\eta^{1/r} \vert \beta \vert^{\frac{r-2}{r}\big( \left(2+\frac{3}{r-2}\right)p(k)-q(k) \big)}.
\]
Hence, if $(2 + \frac{3}{r-2})p(k) - q(k) > 0$ then
\[
\vert \alpha \vert^{2q(k)-p(k)} \vert \beta \vert^{2p(k)-q(k)} \leq \eta^{1/r}.
\]
It remains to note that any $\xi$ such that $\max(\vert \alpha \vert, \eta^{1/r}) < \xi < 1$ works. This completes the proof of Theorem 1.1 (i).

\medskip

\no Let 
\[
 L_{p(k),q(k)}=\tau \circ H_{p(k),q(k)} \circ \tau 
\]
where $\tau(z,w)=(w,z)$. Then
\[
DL_{p(k), q(k)}(0) = \begin{pmatrix}
{\alpha}^{q(k)} \beta^{p(k)} & 0\\
0 & \alpha^{p(k)} \beta^{q(k)}
\end{pmatrix}
\]
and if 
\[
2q(k) - p(k) \ge 0
\]
and $p(k)$ is bounded, a similar calculation applied to $L_{p(k),q(k)}$ shows that the basin of attraction of $\{H_{p(k), q(k)}\}$ is biholomorphic to $\mathbb C^2$.
 
\begin{figure}[H]
 \centering
\begin{subfigure}{.5\textwidth}
 \centering
\begin{tikzpicture}[scale=0.75]
\fill[black!20!white](0.01,0.01) -- (1.5,3) -- (4,3) -- (4,0.01) -- cycle;
\draw[<->](-.5,0)--(4,0);
\draw[<->](0,-.5)--(0,4);
\draw[thick](-.25,-.5)--(2,4);
\node[right] at (4,0) {$p$};
\node[above] at (0,4) {$q$};
\end{tikzpicture}
\caption{$(2 + \frac{3}{r-2})p(k)-q(k) \ge 0$}
\label{pq-1}
\end{subfigure}%
\begin{subfigure}{.5\textwidth}
 \centering
\begin{tikzpicture}[scale=0.75]
\fill[black!20!white](0.01,0.01) -- (0.01,4) -- (3,4) -- (3,1.5) -- cycle;
\draw[<->](-.5,0)--(4,0);
\draw[<->](0,-.5)--(0,4);
\draw[thick](-.5,-.25)--(4,2);
\node[right] at (4,0) {$p$};
\node[above] at (0,4) {$q$};
\end{tikzpicture}
\caption{$(2 + \frac{3}{r-2})q(k)-p(k) \ge 0$}
\label{pq-2}
\end{subfigure}
\caption{}
\end{figure}

\no The shaded regions in Figures (\ref{pq-1}) and (\ref{pq-2}) correspond to those sequences for which the  non--autonomous basins are biholomorphic to $\mbb C^2.$  

\medskip

\end{proof}

%\begin{rem}
%Note that the region $\big(2+\frac{3}{r-2}\big)p(k)-q(k)>0$ converges to $2p(k)-q(k)>0$ as $r \to \infty.$ But if $r$ is fixed as in our case, the regions depicted in Figure (\ref{pq-1})
%and (\ref{pq-2}) can be increased depending on $r$.
%\end{rem}

\begin{proof}[Proof of Proposition \ref{two maps}]
By the given condition, there exists $\tilde{\alpha}$ such that $|\alpha|< \tilde{\alpha}<1$ and $|\beta|>\tilde{\alpha}^k$, and each $F_n$ satisfies
\[ 
\tilde{\alpha}^k\|z\| \le \|F_n(z)\| \le \tilde{\alpha}\|z\|
\]
on a sufficiently small ball $B^k(0;r)$.

\medskip\no 
Now corresponding to the sequence $\{F_n\}$, we will associate another sequence of automorphisms $\{G_n\}$  defined as
$$G_n(z_1,z_2)=
\begin{cases} F(z_1,z_2); \text{ if }F_n=F, \\
(\beta z_1,\alpha z_2);\text{ otherwise. }
\end{cases}$$
 Let $X_0$ be an automorphism of $\mbb C^2$ of the form
\[
X_0(z_1,z_2)=\big(z_1,z_2+ X_{2,0}z_1^k\big) 
\]
where $ X_{2,0} \in \mbb C.$ Also let $\{X_n\}_{n \ge 1}$ be a sequence of polynomial endomorphisms of $\mbb C^2$ defined inductively by
\begin{align}\label{induction} 
X_{n+1}=[F_n \circ X_n\circ G_n^{-1}]_k
\end{align}
where $[\cdot]_k$ means that the degree $(k+1)$ terms are truncated from the expression.  A simple computation by expanding (\ref{induction}) gives
$$X_n(z_1,z_2)=\big(z_1,z_2+ X_{2,n}z_1^k\big)$$ for every $n \ge 1$, i.e., $\{X_n\}_{n \ge 0}$ is a sequence of lower triangular automorphisms of $\mbb C^2.$ It can be checked that

\begin{align*}
  X_{2,n+1}= \beta/\alpha^k X_{2,n}
\end{align*}
 if $F_n=F$ and 
\begin{align*}
  X_{2,n+1}= 1/{\alpha^{k-1}} X_{2,n}+1
\end{align*} 
otherwise. Let $\{\cal{A}_n\}_{n \ge 1}$ be a sequence of affine maps of $\mbb C$ defined as
\[ \cal{A}_{n+1}(z)=\begin{cases}\beta/\alpha^k z; \text { if }F_n=F\\ 1/\alpha^{k-1}z+1; \text{ otherwise. }\end{cases}\]
From \cite{peters}, there exists a $z_0 \in \mbb C$ such that $\cal{A}(n)(z_0)$ is bounded for every $n \ge 1$. If we let $X_{2,0}=z_0$, then $\{X_n\}$ is a bounded sequence of automorphisms of $\mbb C^2.$

\medskip\no 
By Lemma 14 from \cite{Survey}, it follows that the basin of attraction of $\{F_n\}$ at the origin is biholomorphic to the basin of attraction of the sequence $\{G_n\}.$ But the $G_n$'s are upper triangular maps with an attracting fixed point at the origin and hence the basin of attraction of $\{G_n\}$ at the origin is all of $\mbb C^2.$ This completes the proof.
\end{proof}

%%%%%%%%%%%%%%%%%%%%%%%%%%%%%%%%%%%%%%%%%%%

\section{Proof of Theorem \ref{theorem 2}}

	\no For $0 < a< 1$ and $d \ge 2$, consider the sequence of mappings 
	\[F_n(z_1,z_2,\hdots,z_k)=(\eta_n z_k, z_2^d+\eta_n z_1, z_3^d+\eta_n z_2,\hdots, z_k^d+\eta_n z_{k-1})\]
	where $|\eta_n| \le a^{d^n}$ for every $n \ge 0.$ The non--autonomous basin of attraction of this sequence, i.e., $\Omega_{\{F_n\}}$ will be a {\it Short} $\mbb{C}^k$. The arguments used to prove this fact are similar to Forn\ae ss 's proof in \cite{ShortC2}. However, we will briefly rewrite the proof for the sake of completeness.

\medskip\no 	
	 Let  $\Delta^k(0;R)$ denote the polydisk of radius $R$ at the origin and
	\[ F(n)(z):=F_n \circ \cdots \circ F_0(z)\]
	\begin{thm}\label{Short C^k theorem}
	The set $\Omega_{\{F_n\}}$ has the following properties:
	\begin{enumerate}
	\item[(i)] $\Omega_{\{F_n\}}$ is a non--empty open connected set.
	\item[(ii)] $\Omega_{\{F_n\}}=\cup_{j=1}^{\infty}\Omega_j$, $\Omega_j \subset \Omega_{j+1}$, and each $\Omega_j$ is biholomorphic to the unit ball $B^k(0;1)$ in $\mbb C^k.$
	\item[(iii)] The infinitesimal Kobayashi metric on $\Omega_{\{F_n\}}$ vanishes identically.
	\item[(iv)] There exists a plurisubharmonic function $\phi : \mbb C^k \to [\log a , \infty)$ such that
	\[\Omega_{\{F_n\}}=\{z \in \mbb C^k: \phi(z)<0\}.\]
	\end{enumerate}
	\end{thm}
	\begin{proof}
	Let $\Delta^k(0;c)$ denote the polydisk of polyradius $(c,\hdots,c)$, $0<c<1.$ Then 
	\[\| F_n(z)\|_{\infty} \le c^d+|\eta_n| c\] 
for every $z \in \Delta^k(0;c)$, i.e., $F_n(\Delta^k(0;c)) \subset \Delta^k(0;c^d+|\eta_n| c).$ Pick $c'$ such that $0<c<c'<1$ and let $c_l=c(c')^l.$ 
	
	\medskip 
	\no{\it Claim:} For $l \ge 0$, there exists $n$ large enough such that $F_{n+l}(\Delta^k(0;c_l)) \subset \Delta^k(0;c_{l+1}).$ 
	
	\medskip 
	\no Since $l+1 \le d^{l}$ for every $l \ge 0$,
	\begin{align*}
	\log (|\eta_{n+l}|)& = d^{n+l} \log a \le (l+1)d^n \log a \\
	& =\Big(\frac{l+1}{2}\Big) d^n \log a+ \Big(\frac{l+1}{2}\Big) d^n \log a\\
	&< \log c(1-c)+ (l+1) \log c'
	\end{align*}
	for $n$ sufficiently large. So we get 
	\begin{align*}
	|\eta_{n+l}| < c(1-c)(c')^{l+1}<c(c')^{l+1}-c^d(c')^{d^l}=c_{l+1}-c_l^d
	\end{align*}
	i.e.,
	\[c_l^d+c_l |\eta_{n+1}|<c_l^d+|\eta_{n+1}|<c_{l+1}.\]
	Hence the claim is true.
	
	\medskip\no 
	Now define
	\[\Omega_n=\{z \in \mbb C^k: F(n)(z) \in \Delta^k(0;c)\}.\]
	Then $\Omega_{n} \subset \Omega_{n+1}$ for sufficiently large $n$, i.e., for every $n \ge n_0$ and $l\geq 1$, $$F_{n+l} \circ \cdots\circ F_{n+1}(z) \to 0$$ uniformly on $\Omega_n.$ So we have $\cup_{n \ge n_0}\Omega_n \subset \Omega_{\{F_{n}\}}.$ But now note that if $z \in \Omega_{\{F_n\}}$, then $\|F(n)(z)\|<c$ for sufficiently large $n$, i.e., $z \in \Omega_n$ for $n$ large. Hence $\cup_{n \ge n_0}\Omega_n = \Omega_{\{F_{n}\}}.$ This proves statement (i).
	 
	\medskip\no 
	Let $U_n=\{ z \in \mbb C^k: \|F(n)z\|< c\}.$ Then for every $n \ge 0$, $U_n  \subset \Omega_n.$ Note that for $n \ge n_0 \ge 0$ and for every $z \in \Omega_n$, there exists $l \ge 1$ such that 
	\[ F(n+l)(z) \in \Delta^k(0;c) \subset B^k(0;c),\]
	i.e., $\Omega_n \subset U_{n+l} .$ Thus we have
	\[ \Omega=\cup_{n \ge n_0} \Omega_{n} \subset \cup_{m \ge 0} U_{n_0+l+m}\subset \cup_{m \ge 0} \Omega_{n_0+l+m} \subset\Omega,\]
 i.e., $\Omega=\cup_{m \ge 0} U_{n_0+l+m}$, which proves (ii).
 
 \medskip\no 
Pick a point $p \in \Omega$ and  $\xi \in T_p\Omega.$ Let $p_n=F(n)(p)$ and $\xi_n=DF(n)(\xi)$ for every $n \ge 0.$ Note that $p_n \to 0$, hence consequently, $\xi_n \to 0$ as $n \to \infty.$ Fix $R>0$, a sufficiently large value. Now define maps $\eta_n$ from the unit disc as follows:
\[ \eta_n(z)=p_n+zR\xi_n \;\text{for every } z \in \Delta(0;1).\]
Depending on $R>0$, there exists $n_0$ such that $\eta_n(\Delta(0;1)) \subset \Delta^k(0;c)$ for $n \ge n_0.$ Let $\eta=F(n_0)^{-1} \circ \eta_{n_0}.$ Note that $\eta(\Delta(0;1)) \subset \Omega_{n_0}\subset \Omega.$ Thus $\eta(0)=p$ and $\eta'(0)=R \xi.$ Since it is possible to construct a map $\eta$ for every $R>0$, it follows that the infinitesimal Kobayashi metric on $\Omega$ vanishes.

\medskip\no 
	Let $$F(n)(z)=(f_1^n(z),\hdots,f_k^n(z)),$$ i.e, $f_i^n(z)$ denotes the $i-$th component at the $n-$th stage. Define the functions $\phi_n$ as:
	\[ \phi_n(z)=\max\{|f_1^n(z)|,\hdots, |f_k^n(z)|, |\eta_n|\}.\]
	\begin{lem}
	Let $$ \psi_n(z)=\frac{1}{d^n}\log \phi_n(z).$$ Then $\psi_n \to \psi$ where $\psi$ is plurisubharmonic on $\mbb C^k.$
	\end{lem}
	\begin{proof}
	There are two cases to consider:
	
	\medskip\no 
	{\it Case 1:} Let $\phi_n(z) \le 1.$ Since $\eta_{n+1}=\eta_n^d$,
	\begin{align*}
	\phi_{n+1}(z) \le \max\{\phi_n(z)^d+ |\eta_{n+1}|, |\eta_{n+1}|\} \le 2 \phi_n(z)^d.
	\end{align*}
	{\it Case 2:}  Let $\phi_n(z) > 1.$ Then
	\begin{align*}
	\phi_{n+1}(z) \le \max\{\phi_n(z)^d+ |\eta_{n+1}| \phi_n(z), |\eta_{n+1}|\} \le 2 \phi_n(z)^d.
	\end{align*}
	Thus  for every $z \in \mbb C^k$ 
	\[ \frac{1}{d^{n+1}}\log \phi^{n+1}(z) \le \frac{1}{d^{n+1}}\log 2 + \frac{1}{d^n}\log \phi_n(z).\] Now define 
	\[ \Phi_n(z)=\frac{1}{d^n}\log \phi_n(z)+\sum_{j \ge n}\frac{1}{d^{j+1}}\log 2.\] Then $\Phi_n$ is a monotonically decreasing sequence of plurisubharmonic functions and hence its limit will be plurisubharmonic. But note that $\Phi_n \to \psi$ and hence the proof.
	\end{proof}
	\begin{lem}
	$\Omega_{\{F_n\}}=\{z \in \mbb C^k: \psi(z)<0 \}.$
	\end{lem}
	\begin{proof}
	Suppose $\psi(z)<0$, i.e., for large $n>0$ there exists $s<0$ such that
	\[\frac{1}{d^n}\log \phi_n(z)<s.\]
	This implies that $|f_j^n(z)|< e^{d^n s}$ for $1 \le j \le k$ and $n$ sufficiently large or equivalently $F(n)(z) \to 0$ as $n \to \infty.$ Thus $\Omega_{\{F_n\}} \subset \{\psi < 0\}.$ For the other inclusion suppose $ z \in \Omega_{\{F_n\}}.$ Then $\psi_n(z)<0$ for $n$ large which implies that $\psi(z) \le 0.$ Suppose $z(n_0)=F(n_0)^{-1}(0).$ Then $\phi_n(z(n_0))=|\eta_n|$ for every $n \ge n_0.$ Then for sufficiently large $n$,
	\begin{align*}
	 \psi(z(n_0)) &\le \frac{1}{d^n}\log \phi_n(z(n_0))+ \sum_{j \ge n} \frac{1}{d^{j+1}}\log 2\\
	 & =\frac{1}{d^n}\log \eta_n+\sum_{j \ge n} \frac{1}{d^{j+1}}\log 2< \log a+\sum_{j \ge n} \frac{1}{d^{j+1}}\log 2<0.
\end{align*}	
Since $\psi$ is plurisubharmonic and there exists $z(n_0) \in \Omega_{\{F_n\}}$ such that $\psi(z(n_0))<0$ and $\psi(z) \le 0$ for every $z \in \Omega_{\{F_n\}}$, the subaveraging property of plurisubharmonic functions shows that $\psi(z) < 0$ in $\Omega_{\{F_n\}}.$
	\end{proof}
\no	We will need the following observations about the sequence $\{F_n\}$:
\begin{enumerate}
\item[(i)] Except the first one every component of $F_n$ is of the form $z^d+ \eta_n w.$

\medskip\no 
 \item[(ii)] For $z, w \in \mbb C$, there exists $R>0$ such that if $|w|\le|z|$ and $|z|\ge R>1+a$ then for every $n >0$
	 \begin{align}\label{inequality 1}
	 |z^{d}+\eta_n w|>|z|^d-|\eta_n||w|>|z|(|z|^{d-1}-|\eta_n|)>|z| \ge R.
	 \end{align}
\item[(iii)] Each function $F_n$ is a $(k-1)-$composition of $(k-1)-$shift maps, i.e.,
\[ F_n=\underbrace{S_{k-1}^n \circ \cdots\circ S_1^n}_{(k-1)-\text{compositions}}\]
where, for $1 \le i \le k-2$
\[ S_i^n(z_1,\hdots,z_k)=(z_2,z_3,\hdots,z_{k-1},\eta_n z_1+z_2^d) \]
and
\[ S_{k-1}^n(z_1,\hdots,z_k)=(\eta_n^k z_2,z_3,\hdots,z_{k-1},\eta_n z_1+z_2^d).\]
\end{enumerate}
We will show that the filtration properties of shift--like maps in $\mbb C^k$ that were proved in \cite{BP}, extend to our case as well.

\medskip\no 
Let $V=\overline{\Delta^k(0;R)}$ and 
	  \[
	  V_i=\{ z \in \mbb C^k:\|z\|_{\infty}=|z_i|\ge R\}\]for $1 \le i \le k.$ Also let 
	   \[V^+=\bigcup_{i=2}^{k}V_i \text{ and } V^-=V_1.\] 
\no Observe that $V^+$, $V^-$ and $V$ form a disjoint collection where union is all of $\mbb C^k.$
\begin{lem}\label{filtration lemma}
$F_n(V^+) \subset V^+$ for every $n \ge 1.$
\end{lem}
\begin{proof}
Let $z \in V^+.$ Then $z \in V_i$ for some $2 \le i \le k.$ By (\ref{inequality 1})
$$|\pi_i \circ F_n(z)|=|z_i^d+\eta_n z_{i-1}|>|z_i|\ge R.$$ Now 
$$|\pi_1 \circ F_n(z)|=|\eta_n||z_k|<|z_k| \le |z_i|<|\pi_i \circ F_n(z)|.$$
Hence $F_n(z)$ is not contained in $V^-$ or $V$, i.e., $F_n(z) \in V^+.$
\end{proof}
\no Also from the above proof, if $z \in V^+$, then $(F(n))(z) \in V_{i_n}$ for every $n \ge 0$ and $2 \le i_n \le k$, i.e., $\|F(n)(z)\| \ge R.$ Hence $V^+ \cap \Omega_{\{F_n\}}= \emptyset$ which proves that $\Omega_{\{F_n\}}$ is not all of $\mbb C^k.$ 

\medskip\no 
Now we will show that $\psi$ is non--constant. Suppose not. Let $\psi \equiv \alpha<0$ on all of $\Omega_{\{F_n\}}$. Pick a point $z^0 \in \Omega_{\{F_n\}}.$ Choose $r>0$ such that $B^k(z^0,r) \subset \Omega_{\{F_n\}}$ and $\partial B^k(z^0;r) \cap \partial\Omega_{\{F_n\}} \neq \emptyset.$ Let $p \in \partial B(z^0;r) \cap \partial\Omega_{\{F_n\}}.$ Since $p \notin \Omega_{\{F_n\}}$ we know that $\psi(p) \ge 0.$

\medskip\no 
Now for sufficiently small $\ep>0$, the value of $\psi(p)$ should be bounded above by its average on the ball $B^k(p;\ep).$ But now $\psi \equiv \alpha<0$ on $B^k(p ; \ep) \cap \Omega_{\{F_n\}}$ and since $\psi$ is upper semicontinuous, i.e.,
$$\max_{z \in B^k(p;\ep)} \psi(z)<\psi(p)-\alpha$$
i.e.,
\begin{align}\label{pluriharmonic}
\psi(p)  \le \frac{1}{m(B^k(p;\ep))}\Big\{m\big(B^k(p;\ep)\cap \Omega_{\{F_n\}}\big) \alpha+(\psi(p)-\alpha)m\big(B^k(p;\ep) \setminus \Omega_{\{F_n\}}\big)\Big\}
\end{align}
where $m(B^k(p;\ep))$ is the $2k-$dimensional Lebesgue measure of $B^k(p;\ep).$ Now as $B^k(z^0;r) \subset  \Omega_{\{F_n\}}$ and $p \in \partial B^k(z^0;r)$, for $\ep_i \to 0$ 
\[ \frac{m(B^k(p; \ep_i) \cap B^k(z_0;r))}{m(B^k(p;\ep_i))} \to \frac{1}{2}.\] Choose $\ep>0$ such that
\[\frac{1}{2}-\rho \le \frac{m(B^k(p; \ep) \cap B^k(z_0;r))}{m(B^k(p;\ep))}\le \frac{1}{2}+\rho\] where $0<\rho<1/2.$ As $\psi(p)-\alpha>0$, \ref{pluriharmonic} reduces to
\begin{align*}
\psi(p)  &\le \frac{1}{m(B^k(p;\ep))}\Big\{m\big(B^k(p;\ep)\cap \Omega_{\{F_n\}}\big) \alpha+(\psi(p)-\alpha)m\big(B^k(p;\ep) \setminus \Omega_{\{F_n\}}\big)\Big\} \\
&\le \frac{1}{m(B^k(p;\ep))}\Big\{m\big(B^k(p;\ep)\cap B^k(z_0;r)\big) \alpha+(\psi(p)-\alpha)m\big(B^k(p;\ep) \setminus B^k(z_0;r)\big)\Big\} \\
&\le (\frac{1}{2}+\rho)\alpha+(\frac{1}{2}+\rho)(\psi(p)-\alpha) \le \psi(p)(\frac{1}{2}+\rho)< \psi(p).
\end{align*}  
This is a contradiction and hence $\psi$ is non--constant on $\Omega_{\{F_n\}}.$
	\end{proof}
	\no
Thus we have proved that the non--autonomous basin of attraction at the origin of the sequence  $\{F_n\}$ is a {\it Short} $\mbb C^k$.
\medskip\no 

\no Now Theorem \ref{theorem 2} follows as an application of Theorem \ref{Short C^k theorem} and Lemma \ref{filtration lemma}.
\begin{proof}[Proof of Theorem \ref{theorem 2}]
By Theorem \ref{Short C^k theorem}, there exists a positive real number $R>0$ and a {\it Short $\mbb C^k$}, say $\Omega$ such that $\Omega$ is properly contained in $\ov{V_R} \cup int(V_R^-)$, i.e., for every $z=(z_1,z_2,\hdots,z_k) \in \Omega$ either
\[ |z_i| \le R \text{ for every } 1 \le i \le k \text{ or } |z_k| > \max\{R,|z_i|: 1 \le i \le k-1\}.\]
For $1 \le i \le k-1$, let $\phi_i$ be the involution which interchanges the $k-$th coordinate and the $i-$th coordinate and fixes all others, i.e.,
$$\phi_i(z_1,\hdots,z_i, \hdots,z_k)=(z_1,\hdots,z_k,\hdots,z_i).$$ Let
 $$\widetilde{\Omega_i}=\phi_i(\Omega).$$ 
Note that each $\widetilde{\Omega_i}$ is a {\it Short $\mbb C^k$} and $\widetilde{\Omega_i} \subset \ov{V_R} \cup \rm{int}(V_i).$ Also, for a given constant $C \in \mbb C$, let $\cal{A}_C$ denote the affine map of $\mbb C^k$ given by
 \[ \cal{A}_C(z_1,z_2,\hdots,z_k)=(z_1,z_2,\hdots,z_{k-1},z_k+C).\]
For every $1 \le i \le k-1,$ define $\Omega_i$ as
 $$\Omega_i=\cal{A}_{3(i-1)R}(\widetilde{\Omega_i}).$$
\no{\it Claim:} For every $1 \le \alpha,\beta \le k-1$ and $\alpha \neq \beta$, $\Omega_\alpha \cap \Omega_\beta=\emptyset.$

\medskip\no 
Suppose $\Omega_\alpha \cap \Omega_\beta \neq \emptyset.$ Pick $z \in \Omega_\alpha \cap \Omega_\beta $. As $ \alpha\neq \beta$, without loss of generality one can assume that $\alpha > \beta.$ 

\medskip\no 
Since $z \in \Omega_\alpha \cap \Omega_\beta$, it follows that $\cal{A}_{3(\alpha-1)R}^{-1}(z) \in \ov{V_R} \cup \rm{int}(V_\alpha)$ and $\cal{A}_{3(\beta-1)R}^{-1}(z) \in \ov{V_R} \cup \rm{int}(V_\beta).$

\medskip\no 
{\it Case 1:}  If $\cal{A}_{3(\alpha-1)R}^{-1}(z) \in \ov{V_R}$, then
\[ |z_i| \le R \text{ and }|z_k-3(\alpha-1)R| \le R .\] In particular, $|z_\beta| \le R$ and
\begin{align*}
|z_k-3(\alpha-1)R|=|z_k-3(\beta-1)R+3(\alpha-\beta)R|\le  R. 
\end{align*} As $\alpha>\beta$, we have that 
\begin{align*} 
|z_k-3(\beta-1)R|\ge\big||z_k-3(\alpha-1)R|-3(\alpha-\beta)R\big|>R.
\end{align*} 
Hence $\cal{A}_{3(\beta-1)R}^{-1}(z) \notin V_R$, i.e., $\cal{A}_{3(\beta-1)R}^{-1}(z) \in \rm{int}(V_\beta).$
But this means
\[ |z_\beta|>\max\{R,|z_i|,|z_k-3(\beta-1)R|: 1 \le i \le k-1, i\neq \beta \}\] which is a contradiction. Therefore, $\cal{A}_{3(\alpha-1)R}^{-1}(z) \notin V_R.$

\medskip\no 
{\it Case 2:} If  $\cal{A}_{3(\alpha-1)R}^{-1}(z) \in \rm{int}(V_\alpha)$, then
\[ |z_\alpha|>\max\{R,|z_i|,|z_k-3(\alpha-1)R|: 1 \le i \le k-1, i\neq \alpha \}.\] Now if $\cal{A}_{3(\beta-1)R}^{-1}(z) \in \ov{V_R} $ then 
$|z_\alpha| \le R$ which is a contradiction. Hence $\cal{A}_{3(\beta-1)R}^{-1}(z) \in \rm{int}(V_\beta). $ But that will mean
\[ |z_\alpha|>|z_\beta| \text{ and } |z_\beta|>|z_\alpha| \]
which is again a contradiction. Thus the claim follows and this completes the proof.
\end{proof}

 \section{Properties of shift--like maps and proof of Theorem \ref{theorem 1}}
% \no From \cite{HuBu} it is known that there exists a suitable change of coordinates under which the biholomorphic image of a neighbourhood of an algebraic variety of codimension  $2$ (say $\widetilde{V}$) is nice. To state this result precisely, 
\no For $z \in \mbb C^k$, let $z'=(z_1,z_2)$ and $z''=(z_3, \hdots, z_k).$ Let $\widetilde{V}$ be be an algebraic variety of codimension $2$ in $ \mbb C^k$, $k \ge 3$ . For $\ep>0$, let
	\[ A_{\ep}=\{z \in \mbb C^k: \|z'\|_{\infty}<\ep \}\]
	and  
	\[B_{\ep}=\{z \in \mbb C^k: \|z'\|_{\infty}< \ep \|z''\|_{\infty}\}.\]
	Being algebraic there exists (see \cite{HuBu}, \cite{chirka}) a complex linear map $L_{\ep}$ such that $L_{\ep}(\widetilde{V}) \subset A_{\ep} \cup B_{\ep}.$
	
	\medskip\no 

\begin{prop}\label{pre theorem 1}
Let $\widetilde{V} \subset \mbb C^k$, $k \ge 3$ be an an algebraic variety of codimension $2.$ Then there exists a short $\mbb C^k$ that avoids $\widetilde{V}.$	
\end{prop}
\begin{proof}%[Proof of Theorem \ref{theorem 1}]
Let $V$, $V^+$, $V^-$ and $R \gg 1$ be as obtained in the proof of Theorem \ref{Short C^k theorem}. There exists a change of coordinates $L_{\ep}$ such that $L_{\ep}(\widetilde{V}) \subset A_{\ep} \cup B_{\ep}$ with $0<\ep<R^{-1}.$ Now we further compose this map with an affine map 
\[\cal{A}_R(z_1,z_2,\hdots,z_k)=(z_1,z_2,\hdots,z_k)+(0,2R,0,\hdots,0).\]

\no {\it Claim:} $\cal{A}_R \circ L_{\ep}(\widetilde{V}) \subset V^+.$ 

\medskip\no 
It is enough to show that $\cal{A}_R(A_{\ep} \cup B_{\ep}) \subset V^+$.  Pick $z \in \cal{A}_R(A_{\ep} \cup B_{\ep}).$ Then 
\[ (z_1,z_2-2R,\hdots,z_k) \in A_{\ep} \cup B_{\ep} \]
and for this point $z'=(z_1,z_2-2R)$ and $ z''=(z_3, \hdots,z_k).$

\medskip\no 
{\it Case 1:} If $z' \in A_\ep$ then
\[ |z_2-2R|\le \ep \text{ and } |z_1|\le \ep ,\]
i.e., $|z_1|<|z_2|$ and $|z_2|>R$. Hence $z \in V^+.$ 

\medskip\no 
{\it Case 2:} If $z \in B_{\ep}$, then two cases arise. Either $|z''|\le R$, i.e., $|z'|<1$ and a similar analysis as above shows that $z \in V^+$. Otherwise, for some $3 \le i \le k$ if $|z_i|>R$, then
\[ |z_1| \le \ep |z_i|<|z_i|.\] Hence $z \notin V\cup V^-$, i.e., $z \in V^+.$
 
\medskip\no 
Thus for $k\geq 3$, $L_{\ep}^{-1} \circ \cal{A}_R^{-1}(\Omega_{\{F_n\}})$ is a {\it Short $\mbb C^k$} which does not intersect the algebraic variety $\widetilde{V}.$
\end{proof}
\begin{rem}
Note that Proposition \ref{pre theorem 1} does not apriori ensure that the {\it Short $\mbb C^k$} contains a Fatou--Bieberbach domain. The known technique to construct a {\it Short $\mbb C^2$} containing a Fatou--Bieberbach domain is to look at sub--level sets of the positive Green's function of a H\'{e}non map with an attracting fixed point. 
\end{rem}
\no To prove Theorem \ref{theorem 1}, we will study some properties of sub--level sets of Green's function associated with a shift--like maps.
%\section{Proof of Theorem \ref{theorem 2}}
%\no Consider the $(k-1)-$shift map in $\mbb C^k$ of the form 
%\[S(z_1,\hdots,z_k)=(z_2,\hdots,z_k, \delta(z_2^d-z_1))\]
%for some $0<|\delta|<1$ and $d \ge 2.$ Then by \cite{BP} the positive Green's function for the shift map is:
%\[ G^{+}(z)=\lim_{n \to \infty} \frac{\log^{+}{\|S^{ (k-1) n}(z)\|}}{d^n}.\]

\medskip
\no Let $S: \mbb{C}^k \to \mbb C^k$ be a polynomial shift--like map of degree $d$ and type $\nu.$ From \cite{BP} we know that the positive Green's Function associated to $S$ is defined as:
	\[ G^{+}(z)=\lim_{n \to \infty} \frac{1}{d^n}\log^{+}{\|S^{ \nu n}(z)\|}\]
	and the negative Green's Function associated to $S$ is
	\[ G^{-}(z)=\lim_{n \to \infty} \frac{1}{d^n}\log^{+}{\|S^{ -(k-\nu) n}(z)\|}.\]
%Let $V_i$ for $1 \le i \le k$ be the subsets of $\mbb C^k$ as defined earlier, i.e., for sufficiently large $R>0$
%\begin{align*}
%V_i=\{ z \in \mbb C^k:\|z\|_{\infty}=|z_i|\ge R\}.
%\end{align*}
Recall from \cite{BP} that
\[V_R=\Delta^k(0;R),\; V^+_R=\bigcup_{i=k-\nu+1}^k V_i\; \text{and}\; V^-_R=\bigcup_{i=1}^{k-\nu}V_i\] 
is a filtration and the set of zeros of the positive Green's function is contained in the union of $V_R$ and $ V^+_R,$ i.e.,
\[\{z \in \mbb C^k:G^+(z)=0\}=\{z \in \mbb C^k: S^n(z) \text{ is bounded as } n \to \infty\} \subset V_R \cup V^-_R.\]
\begin{prop}\label{Green}
Let 
\[ S(z_1,z_2,\hdots,z_k)=\big(z_2, \hdots,z_k, \delta(z_{k-\nu+1}^d-z_1)\big)\]
be a shift--like automorphism of $\mbb C^k$ of type $\nu$ and degree $d \ge 2$
where $\delta \in \mbb C^*.$ Then for every $c>0$ there exists $R>0$ such that for every $z \in V^+_R$, $G^+(z)\ge c$, i.e.,
\[ \{z \in \mbb C^k: G^+(z) < c\} \subset V_R \cup V^-_R.\]
\end{prop}
\begin{proof}
The $(k-\nu)-$th iterate of $S$ is:
\[ S^{k-\nu}(z_1,z_2,\hdots,z_k)=\big(z_{\nu+1},\hdots,z_k, \delta(z_{k-\nu+1}^d-z_1),\hdots,\delta(z_k^d-z_{\nu})\big).\]
Suppose $z \in V^+.$ Then $z \in V_{i_0}$ for $k-\nu+1 \le i_0 \le k.$ For sufficiently large $R>0,$
\begin{align}\label{Inequality_green}
 \nonumber |\pi_{i_0}(S^{k-\nu}(z))|&=|\delta|(|z_{i_0}|^d-|z_{i_0-1}|)  \ge |\delta|(|z_{i_0}|^{d}-|z_{i_0}|) \\
 &\ge \frac{|\delta|}{2}|z_{i_0}|^d>|z_{i_0}|^{d-1}> \|z\|_{\infty}>R.
 \end{align}
% \begin{align}\label{greater-inequality 1}
%  |\pi_{i_0}(S^{k-1}(z))|>|z_{i_0}|^{d-1}> \|z\|_{\infty}>R.
%  \end{align} 
%  Let $C_R$ denote the following constant which is less than $1$ for sufficiently large values of $R.$
% \begin{align}\label{constant}
%  C_R=\frac{|\delta|(R^{d-1}-1)}{R^{d-1}}<1.
%  \end{align}
  This can be rewritten as
 \begin{align}\label{greater-inequality}
 \|S^{k-\nu}(z)\|_{\infty} \ge|\pi_{i_0}(S^{k-\nu}(z))|\ge \frac{|\delta|}{2} \|z\|_{\infty}^{d}.
 \end{align}
 {\it Claim:} For every $n \ge 1$
 \begin{align}\label{main_green}
 \|S^{n(k-\nu)}(z)\|_{\infty} \ge \Big(\frac{|\delta|}{2}\Big)^{\sum_{i=0}^{n-1}d^{i}} \|z\|_{\infty}^{d^n}.
 \end{align}
 It is clear from the above calculations that the claim is true when $n=1.$ Now assume that (\ref{main_green}) is true for some $n.$ We will show that it is true for $n+1.$ Since $\frac{R |\delta|}{2} >1$ and $\|z\|_{\infty}>R$,
 \begin{align*}
 \|S^{n(k-\nu)}(z)\|_{\infty} \ge \Big(\frac{|\delta|}{2}\Big)^{\sum_{i=0}^{n-1}d^{i}} \|z\|_{\infty}^{d^n}> \|z\|_{\infty}>R
 \end{align*}
 i.e., $S^{n(k-\nu)}(z) \in V_{i_n}$ for some $1 \le i_n \le k.$ But now for every $1 \le j \le k-\nu$ $$|\pi_j \circ S^{n(k-\nu)}(z)| \le \| S^{(n-1)(k-\nu)}(z)\|_{\infty}$$ whereas from  (\ref{Inequality_green}) and (\ref{greater-inequality}), it follows that 
 \[\|S^{n(k-\nu)}(z)\|_{\infty}> \|S^{(n-1)(k-\nu)}(z)\|_{\infty},\] hence $S^{n(k-\nu)}(z) \in V_{i_n}$ for some $k-\nu+1 \le i_n \le k.$ Now from (\ref{Inequality_green}),
 \begin{align*}
 |\pi_{i_n}(S^{(n+1)(k-\nu)}(z))| &\ge \Big(\frac{|\delta|}{2}\Big) |\pi_{i_n}(S^{n(k-\nu)}(z))|^d \\
 &\ge \Big(\frac{|\delta|}{2}\Big) \Big\{\Big(\frac{|\delta|}{2}\Big)^{\sum_{i=0}^{n-1}d^{i}} \|z\|_{\infty}^{d^n}\Big\}^d \\
 & \ge \Big(\frac{|\delta|}{2}\Big)^{\sum_{i=0}^{n}d^{i}} \|z\|_{\infty}^{d^{n+1}}.
 \end{align*}
 Hence the claim is true.
 
 \medskip \no 
 Now there exists $0 < \tilde{\delta} \le 1$ such that (\ref{main_green}) can be further modified as
 \begin{align}\label{main_green 1}
 |\pi_{i_n}(S^{n(k-\nu)}(z))| \ge \tilde{\delta}^{\big(\sum_{i=0}^{n-1}d^{i}\big )} \|z\|_{\infty}^{d^{n}}.
 \end{align}
 \medskip\no 
 Since $\sum_{i=0}^{n-1}d^i <d^n$  for every $n \ge 1$ and $0<\tilde{\delta} \le 1$,
 (\ref{main_green 1}) can be modified as
 \begin{align}\label{modified-green}
 \|S^{n(k-1)}(z)\|_{\infty} > \big(\tilde{\delta}\|z\|_{\infty} \big)^{d^n}.
 \end{align}
 It follows from (\ref{modified-green}) that
\[ G^{+}(z)=\lim_{n \to \infty} \frac{1}{d^n}\log^{+}{\|S^{ (k-1) n}(z)\|}  \ge \tilde{\delta}\|z\|_{\infty} .\]
Now $R$ can be appropriately modified such that $\tilde{\delta}\|z\|_{\infty} > c$ (here $c$ is the given constant) for every $z \in V_R^+$, i.e.,
\[\{z \in G^+(z)>c\} \supset V_R^+\] or equivalently\[\{z \in G^+(z)<c \}\subset {V_R} \cup  V_R^- .\]
\end{proof}
\no 
As a corollary of Proposition \ref{Green}, Proposition \ref{pre theorem 1} we prove the following result:
\begin{thm}\label{pre theorem 2}
	For $k\geq 3$, let 
	 \[S(z_1,\hdots,z_k)=(z_2,\hdots,z_k, \delta(z_2^d-z_1))\]
be a shift--like automorphisms of type $k-1$ in $\mbb C^k$ for some $\delta \in \mbb C^*$ and $d \ge 2.$ For a given $c>0$ and an algebraic variety $\widetilde{V}$ of codimension 2, there exists an appropriate change of coordinates such that the $c-$sublevel set of the positive Green's function, i.e.,
\[\{z \in \mbb C^k: G^+(z)<c\}\]
does not intersect $\widetilde{V}$ in the new coordinate system.
\end{thm}
%\begin{proof}%[Proof of Theorem \ref{theorem 2}]
%\end{proof}
%\begin{rem}
%Note that Proposition \ref{Green} is actually true for any polynomial shift--like map in $\mbb C^k$ of degree at least 2.
%\end{rem}
% \begin{rem}
%In Theorem \ref{pre theorem 2}, if $0<|\delta|<1$, then for any $c>0$ the $c-$sublevel set of the positive Green's function contains a Fatou--Bieberbach domain. However, it might not be a {\it Short $\mbb{C}^k$} since there is no analogue results about the connectivity of the sets $K^+$ for shift--maps.
%\end{rem}
\begin{lem}\label{condition on eta}
Let $\{F_n\}$ be the sequence of automorphisms as in Theorem \ref{Short C^k theorem}. If there exists $M>1$ such that $|\eta_{n+1}|< M |\eta_n|^d$, where $|M \eta_0|<1$, then $\Omega_{\{F_n\}}$ is a {\it Short $\mbb C^k$}.
\end{lem}
\begin{proof}
Let $M|\eta_0|=\alpha<1$. Note that $|\eta_0|< M|\eta_0|$ and $|\eta_1| \le  M|\eta_0|^d.$

\medskip\no 
{\bf Induction statement:} For every $n \ge 1$, $ |\eta_n|  \le M^{(1+d+\hdots+d^{n-1})}|\eta_0|^{d^n}.$

\medskip\no 
The above statement is true for $n=1.$ So assume it is true for some $n.$ Then
\[ |\eta_{n+1}| \le M |\eta_n|^d \le M\big(M^{1+d+\hdots+d^{n-1}}|\eta_0|^{d^n}\big)^d \le M^{1+d+\hdots+d^{n}}|\eta_0|^{d^{n+1}}.\]
\no 
From the induction statement it follows that
\[ |\eta_n| < (M|\eta_0|)^{d^n}=\alpha^{d^n}.\]
Hence from Theorem \ref{Short C^k theorem}, it follows that $\Omega_{\{F_n\}}$ is a {\it Short $\mbb C^k$}.
\end{proof}
\no Now we can complete the proof of Theorem \ref{theorem 1}.
\begin{proof}[Proof of Theorem \ref{theorem 1}.]
Choose $0<a<1$ and let $F_n$ be a sequence of automorphisms of $\mbb C^k$ defined as follows:
\begin{align*}
F_n(z_1,z_2,\hdots,z_k)&=\big(\eta_nz_k, z_2^2+\eta_n z_1, \hdots, z_k^2+\eta_n z_{k-1}\big)
\end{align*}
where $\eta_n=a^{{2^n}+1}.$ Then 
\[ \eta_{n+1}=\frac{1}{a} \eta_n^2 \text{ and } a\eta_0=a<1.\] From Lemma \ref{condition on eta},  it follows  that $\Omega_{\{F_n\}}$ is a {\it Short $\mbb{C}^k$}. Moreover from Proposition \ref{pre theorem 1}, for a given algebraic variety ${V}$ of codimension 2, there exists an appropriate linear change of coordinates (say $L$) of $\mbb C^k$ such that $L(\Omega_{\{F_n\}})$ does not intersect ${V}.$

\medskip\no 
{\it Claim:} $\Omega_{\{F_n\}}$ contains a Fatou--Bieberbach domain.

\medskip\no 
Consider the $(k-1)-$shift--like automorphism of $\mbb C^k$ given by 
$$F(z_1,z_2,\hdots,z_k)=(a z_k,z_2^2+az_1,\hdots, z_k^2+az_{k-1}).$$ Clearly the basin of attraction of $F$ at the origin (say $\Omega_F$) is a Fatou--Bieberbach domain by Rosay--Rudin (\cite{RR1}). 

\medskip\no 
For a given constant $C \neq 0$, let $l_C$ denote the linear map from $\mbb C^k$ to $\mbb C^k$ given by
\[ l_C(z_1,z_2,\hdots,z_k)=(Cz_1,Cz_2,\hdots, Cz_k).\]
Note that 
\[ F_n(z_1,z_2,\hdots,z_k)=l_{a^{2^{n+1}}} \circ F \circ l_{a^{-2^{n}}}(z_1,z_2,\hdots,z_k). \]
Then for $z \in l_a(\Omega_F)$, there exists $n$ sufficiently large such that
$\|F(n)(z)\|< a^{2^{n+1}}$, i.e.,
\[ \psi(z)<\log a < 0.\] Thus $l_a(\Omega_F)$ is a Fatou--Bieberbach domain contained in $\Omega_{\{F_n\}}$. The {\it Short} $\mbb C^k$ and Fatou--Bieberbach domain claimed in Theorem \ref{theorem 1} are then $L(\Omega_{\{F_n\}})$ and $L(l_a(\Omega_F))$ respectively.
\end{proof}
\section{Controlling the boundary of a {\it Short} \texorpdfstring {$\mathbb{C}^k$}{} on a fixed polydisk}
\no In this section, we will construct a {\it Short} $\mbb C^k$, $k \ge 3$ with some control on its boundary on a fixed polydisk. Theorem \ref{pre theorem 3} is the main statement here -- it shows the existence of a {\it Short} $\mbb C^k$, whose boundary is very close to $(k-1)-$faces of the polydisk. In addition, this {\it Short} $\mbb C^k$ is almost a cylinder along the remaining direction. 
We will use some ideas from \cite{Gl2}, \cite{Gl} and \cite{console}. 

\medskip\no 
First recall the following lemma from \cite{Gl} and \cite{console}.
\begin{lem}\label{smooth_function}
For a given $R\gg 1$, $\ep>0$ and $l \ge 1$, there exists $\alpha_0(\ep,R)>0$ such that if $0 \le \alpha \le \alpha_0$, then
\[ \big\{ (z,w)\in \mbb{C}^2:|z^2+\alpha w|=1 , |w| \le R\big\}=\big\{ (\phi_{\alpha}(\xi,w)\xi, w): \xi \in \partial \Delta, \phi_{\alpha} \in {C}^l\big(\partial \Delta \times \ov{\Delta(0;R)}\big)\big \}\]where $\phi_{\alpha} \in  {C}^l\big(\partial \Delta \times \ov{\Delta(0;R)}\big).                                                                                                                                                                                                                                                                                                                                                                                                                      $ Moreover, $\| \phi_{\alpha}-{1}\|_{ {C}^l(\partial \Delta \times \ov{\Delta(0;R)})} < \ep.$
\end{lem}
\no A proof of this for $l=1$ can be found in \cite{Gl}. For $l > 1$, this was observed in \cite{console} -- See Lemma 5.3.4 therein.

\medskip
\no It follows that if $0 < \alpha \le \alpha_0(\ep,R)$, then
{\small \[ \big\{ (z,w)\in \mbb{C}^2:|z^2+\alpha w|\le 1 , |w| \le R\big\}=\big\{ (t\phi_{\alpha}(\xi,w)\xi, w): 0 \le t\le 1, \xi \in \partial \Delta, \phi_{\alpha} \in  {C}^l(\partial \Delta \times \Delta(0;R))\big \}. \]}
\no 
As in Section 2, let
\begin{align}\label{automorphism}
F_{\alpha}(z_1,z_2,\hdots,z_k)=(\alpha z_k, z_2^d+\alpha z_1, z_3^d+\alpha z_2,\hdots, z_k^d+\alpha z_{k-1})
\end{align}
which is an automorphism of $\mbb C^k$ for any $\alpha >0$. For a sequence of automorphisms $\{F_i\}$ of $\mbb C^k$, let $F(n)$ and $F(m,n)$ denote the following maps
\[ F(n)(z)=F_n \circ \cdots \circ F_0(z) ,\; F(m,n)=F_{m+1}\circ \cdots\circ F_n=F(m)^{-1}\circ F(n)\] where $0 \le m \le n.$ Also for a given $\ep>0$ and a compact subset (say $K$) of $\mbb{C}^k$ we will denote the $\ep-$tube around $K$ by $N_{\ep}(K)$, i.e.,
\[ N_{\ep}(K)=\{z \in \mbb C^k: \text{dist}(z,K)<\ep\}.\]
\begin{lem}\label{neighbourhood}
For a given $R \gg 1$, $\ep>0$  and $l \ge 1$ there exists $\alpha_0(\ep,R)>0$ such  that for every $0 < \alpha \le \alpha_0$ there exists $\phi_{\alpha} \in  {C}^l\big(\partial \Delta \times \ov{\Delta(0;R)}\big)$ with the following properties:
\begin{itemize}
\item[(i)] Let $(z_1,z_2,\hdots,z_k) \in F_{\alpha}^{-1}\big(\Delta^k(0;1)\big)\cap \ov{\Delta^k(0;R)}$.  Then for every $ 2 \le i \le k$, the value of $z_i$ depends on $z_{i-1}$ recursively in the following way:  for every $z_1 \in \ov{\Delta(0;R)}$ and $2 \le i \le k$
\begin{align*}
z_i=t_i \phi_\alpha(\xi_i,z_{i-1}) \xi_i,
\end{align*}
where $t_i \in [0,1)$, $\xi_i \in \partial \Delta$.

\medskip
\item[(ii)] $F_{\alpha}^{-1}\big(\partial\Delta^k(0;1)\big)\cap \ov{\Delta^k(0;R)} \subset N_{k\ep}\big(\ov{\Delta(0;R)} \times \partial\Delta^{k-1}(0;{1})\big).$
\end{itemize}
\end{lem}
\begin{proof}
Let $\alpha_0'>0$ such that $\alpha_0'R< 1$. If $0 < \alpha< \alpha_0'$ then
\begin{align}\label{1}
 F_{\alpha}^{-1}\big(\Delta^k(0;1)\big) \cap \ov{\Delta^k(0;R)}=\big\{ (z_1,\hdots,z_k) \in \ov{\Delta^k(0;R)}: |z_i^2+\alpha z_{i-1}| < 1 \text{ for } 2 \le i \le k\big\}.
\end{align} 
\no From Lemma \ref{smooth_function}, there exists $\alpha_0''$ such that for every $0 < \alpha\le \min\{\alpha_0', \alpha_0''\}$
\[ z_i=t_i\phi_{\alpha}(\xi_i, z_{i-1}) \xi_i\] if $|z_{i-1}| \le R$ for every $2 \le i \le k$. Here $t_i \in [0,1)$, $\xi_i \in \partial \Delta$, $\phi_{\alpha} \in  {C}^l\big(\partial \Delta \times \ov{\Delta(0;R)}\big)$ and $$\|\phi_{\alpha}-{1}\|_{ {C}^l(\partial \Delta \times\ov{ \Delta(0;R)})} < \ep .$$ Since $|z_1| \le R$, $|z_i| \le {1}+\ep < R$ for every $2 \le i \le k$. Hence (i) follows.

 \medskip\no 
For a fixed $i$, $2 \le i \le k$, define
 \[ H_{i,\alpha,1}=\{(z_1,\hdots,z_k): |z_i^2+\alpha z_{i-1}|=1 \text{ and } |z_j^2+\alpha z_{j-1}| \le 1, |z_1| \le R \text{ for } 2 \le j \le k, j \neq i  \}. \] 
From (\ref{1}), the set $F_{\alpha}^{-1}\big(\partial\Delta^k(0;1)\big) \cap \ov{\Delta^k(0;R)}$ can be realized as
\[ F_{\alpha}^{-1}\big(\partial\Delta^k(0;1)\big) \cap \ov{\Delta^k(0;R)}= \bigcup_{i=2}^k H_{i,\alpha,1}. \]
Also $\ov{\Delta(0;R)} \times \partial\Delta^{k-1}(0;{1})$ can be written as the union of its faces, i.e.,
\[\ov{ \Delta(0;R)} \times \partial\Delta^{k-1}(0;{1})= \bigcup_{i=2}^{k} D_{i, {1}} \]
where 
\[ D_{i, {1}}=\{(z_1,\hdots,z_k): |z_i|={1}, |z_j| \le {1} \text{ and } |z_1| \le R \text{ for } 2 \le j \le k, j \neq i \}\]for a fixed $i$, $2 \le i \le k$. 
Now from the bound on $\phi_{\alpha}$ it follows that the distance between $H_{i, \alpha,1}$ and $D_{i,{1}}$ is less than $k\ep$. Thus (ii) follows.
\end{proof}
\begin{rem}\label{Remark 1}
The function $\phi_{\alpha}$ obtained in Lemma \ref{smooth_function} is actually a positive smooth function with 
\[ {1}< {\phi_{\alpha}}_{|\partial \Delta \times \ov{\Delta(0;R)}} \le 1+\ep.\]
\end{rem}
\begin{rem}\label{Remark 2}
Using this, the conclusion of Lemma \ref{neighbourhood} can be improved slightly, i.e.,
$$F_{\alpha}^{-1}\big(\Delta^k(0;1)\big)\cap \ov{\Delta^k(0;R)} \subset \big(\ov{\Delta(0;R)} \times \Delta^{k-1}(0;{1}+k\ep)\big).$$
\end{rem}
\no The next result is Lemma 5.3.5 from \cite{console}. We will include the proof for the sake of completeness. Before stating the result, we introduce certain notations. Suppose $F_n=F_{\alpha_n}$ is a sequence of automorphisms as in (\ref{automorphism}). For $n \ge 1$ and $0<c \le 1$, let 
\[ \Omega_{n,c}=\{z \in \mbb C^k: F(n)(z) \in \Delta^k(0;c) \}\] and \[ \Omega=\bigcup_{n=1}^{\infty} \Omega_{n,c} \text{ for } 0<c<1.\]
From Theorem \ref{theorem 1} we know that $\Omega$ is a {\it Short $\mbb C^k$} if $\alpha_{n+1} \le \alpha_{n}^2$ and $0 < \alpha_0 < 1$ and it is the non--autonomous basin of attraction of the sequence $\{F_n\}$ at the origin.
\begin{lem}\label{topological}
Fix $0<c<1$. For a given compact connected set $K$ and $\ep>0$, there exists $\tau=\tau(n,\ep,K)>0$ such that  if $1-\tau<c<1$ then $\partial\Omega_{n,c} \cap K  \subset N_{\ep}\big(\partial \Omega_{n,1} \cap K \big). $
\end{lem}
\begin{proof}
Let $P_1=F(n)(\partial\Omega_{n,1})$ and  $V_{\ep}= F(n) \big(N_{\ep}(\partial \Omega_{n,1})\big).$ Note that 
\[ V_{\ep}=\bigcup_{z \in \partial \Omega_{n,1}} F(n)(B^k(z;\ep)) ,\] and that $V_{\ep}$ is an open cover of $P_1.$ Then for every $z \in \partial \Omega_{n,1} \cap K$ there exists $r_z>0$ such that $B^k(F(n)(z);r_z) \in V_{\ep}.$ The collection $B^k(F(n)(z);r_z/k)$ as $z$ varies in $\Omega_{n,1}$ forms an open cover of $P_1$.
% defined as
%\[ W_{\ep}=\bigcup_{z \in \partial \Omega_{n,1} \cap K} B_{\frac{r_z}{k}}(F(n)(z)). \] 
Since $P_1$ is compact there exists $N_0 \ge 1$ such that
\[ P_1 \subset \bigcup_{i=1}^{N_0} B^k(F(n)(z_i);{{r_{z_i}}/{k}}) \subset V_{\ep}. \]
Let $\tau_1 =\min \Big\{\frac{r_{z_i}}{k}: 1 \le i \le N_0 \Big\}$. Let 
\[ P_c=F(n)(\partial \Omega_{n,c}).\] If $1-\tau_1 < c< 1$, the distance between $P_c$ and $P_1$ is at most $k\tau_1 < \min\{ r_{z_i}:1 \le i \le N_0 \}.$ Thus $P_c \subset V_{\ep}$ and 
$\partial\Omega_{n,c} \subset N_{\ep}\big(\partial \Omega_{n,1} \big). $

\medskip\no 
Now by the connectedness of $K$, there exists $\tau_2>0$ such that if $1-\tau_2< c< 1$, the distance between $P_c \cap F(n)(K)$ and $P_1 \cap F(n)(K)$ is at most $k\tau_2$. Let $\tau=\min\{\tau_1,\tau_2\}.$ Hence, for every $c$ such that $1-\tau < c< 1$,
$$\partial\Omega_{n,c} \cap K \subset N_{\ep}\big(\partial \Omega_{n,1} \cap K \big).$$
\end{proof}

\no Let $n \ge 0$. Suppose there exist $(n+1)-$ real positive constants $\{\alpha_i:0 \le i \le n\}$ such that $\alpha_{i+1} \le \alpha_i^2$ for every $0 \le i < n$. Then this finite collection can be extended to a infinite sequence $\{\alpha_m\}$ such that
$\alpha_{m+1} \le \alpha_m^2$ for every $m \ge 0$. This extension is evidently not unique. However, the basin of attraction of the sequence of automorphisms $\{F_m\}$, where $F_m = F_{\alpha_m}$ is a {\it Short} $\mbb C^k$. Here, we will show that if there exist $(n+1)-$automorphisms $\{F_i: 0 \le i \le n\}$ such that we can control the following:

\begin{enumerate}
\item [(i)] the behaviour of $\Omega_{n,1}$ on a large polydisk,
\item[(ii)]  the behaviour of $\Omega_{i,c_i}$ for a collection of increasing real constants $\{c_i: 0 \le i \le n\}$, and
\item[(iii)] the behaviour of $\Omega$, where $\Omega$ is the basin of attraction of the sequence $\{F_m\}$ obtained by some appropriate extension of the collection $\{\alpha_i: 0 \le i \le n\}$,
\end{enumerate}

then the finite collection of automorphisms can be appended with $F_{n+1}$ such that the collection $\{F_i:0 \le i \le n+1\}$ will also satisfy the above three properties. Essentially our target is to show that we can inductively control appropriate these domains. This phenomenon is stated in the following result:

\begin{prop}\label{main proposition}
Suppose for some $n \in \mbb{N}$, there exist automorphisms $\{F_{\alpha_i}: 0 \le i \le n\}$ of $\mbb{C}^k$, with the following properties.
\begin{itemize}
\item[(a)] For a given  $R \gg 1$
\[  \Omega_{n,1} \cap \ov{\Delta^k(0;R)}=\{ (z_1,z_2,\hdots,z_k) \in \ov{\Delta^k(0;R)}: |z_i^n|<1 \text{ for every } 2 \le i \le k \}. \]
\item[(b)] For every $0 \le i \le n$, there exist $(n+1)-$increasing constants $0 < c_i < 1$ and $(n+1)-$sequences $\{\alpha_m(i)\}_{m \ge n+1}$ $$\big(\text{i.e.,} \{\alpha_{n+1+j}(0)\}_{j=0}^{\infty} ,\{\alpha_{n+1+j}(1)\}_{j=0}^{\infty}, \hdots, \{\alpha_{n+1+j}(n)\}_{j=0}^{\infty}\big),$$ such that if $F_m=F_{\alpha_m}$ where $$0<\alpha_m \le \min\{\alpha_m(i),\alpha_{m-1}^2: 0 \le i \le n\}$$ for $m \ge n+1$,  then
\[ \Omega_{i,c_i}\subset \Omega_{i+1, c_{i+1}} \subset \Omega_{m,c_n} \subset \Omega  \text{ for } 0 \le i \le n-1 \]
where $\Omega$ is the basin of attraction of $\{F_j \}_{j=0}^{\infty}$ at the origin.
\end{itemize}
Then for a given $\ep>0$, there exists $0<c_{n+1}(\ep)<1$ and $\alpha_{n+1}'>0$ such that
\begin{itemize}
\item[(i)] For $F_{n+1}=F_{\alpha_{n+1}}$ where $0< \alpha_{n+1}< \alpha_{n+1}'$, 
\[  \Omega_{n+1,1} \cap \ov{\Delta^k(0;R)}=\{(z_1,z_2,\hdots,z_k) \in \ov{\Delta^k(0;R)}:|z_i^{n+1}|<1 \text{ for every } 2 \le i \le k\}\]
and
\[\partial\Omega_{n+1,1}\cap \ov{\Delta^k(0;R)}\subset N_{\ep}\big(\partial \Omega_{n,1} \cap \ov{\Delta^k(0;R)}\big).\]
\item[(ii)] $\partial \Omega_{n+1,c_{n+1}} \cap \ov{\Delta^k(0;R)}$ is contained in the $\ep-$neighbourhood of $\partial \Omega_{n+1,1} \cap \ov{\Delta^k(0;R)}$, i.e.,
\[\partial \Omega_{n+1,c_{n+1}} \cap \ov{\Delta^k(0;R)}\subset N_{\ep}\big(\partial \Omega_{n+1,1} \cap \ov{\Delta^k(0;R)}\big) \text{ and } c_{n+1} \ge c_n.\]
\item[(iii)] There exists a sequence of positive real numbers $\{\alpha_m(n+1)\}_{m \ge n+2}$ such that if $F_m=F_{\alpha_m}$ for every  $m \ge n+2$ where $$0 < \alpha_m \le \min\{\alpha_{m-1}^2, \alpha_m(i): 0 \le i \le n+1\}$$ and $\Omega$ is the basin of attraction of $\{F_j\}_{j=0}^{\infty}$ at the origin, then  
\[ \Omega_{i,c_{i}}\subset \Omega_{i+1,c_{i+1}} \subset \Omega_{m, c_{n+1}} \subset \Omega \text { for every }0 \le i \le n.\]
\end{itemize}
\end{prop}

\begin{proof}
Fix $R_n>1$ such that $F(n)\big( \ov{\Delta^k(0;R)}\big)\subset\subset \ov{\Delta^k(0;R_n)}.$ By continuity of $F(n)^{-1}$, for $\ep>0$ there exists $\delta>0$ such that for $z,z' \in \ov{\Delta^k(0;R_n)}$, 
\[  \|F(n)^{-1}(z)-F(n)^{-1}(z') \|< \ep \] whenever $\|z-z'\|< \delta/k.$ 

\medskip\no 
By Lemma \ref{neighbourhood}, there exists $\alpha_{n+1}'=\alpha_{n+1}'(\delta,R_n)$ such that if $0<\alpha_{n+1} \le \alpha_{n+1}'$ and $F_{n+1}=F_{\alpha_{n+1}}$, then
{\small \begin{align*}
 F_{n+1}^{-1}\big(\Delta^k(0;1)\big) \cap \ov{\Delta^k(0;R_n)}=\big\{ (z_1,\hdots,z_k) \in \ov{\Delta^k(0;R_n)} : |z_i^2+\alpha_{n+1} z_{i-1}| < 1 \; \text{for } 2 \le i \le k, \;|z_1| \le R_n \big\} 
 \end{align*}
 }
 and 
 \begin{align}\label{2}
 F_{n+1}^{-1}(\partial\Delta^k(0;1))\cap \ov{\Delta^k(0;R_n)} \subset N_\delta\big(\ov{\Delta(0;R_n)} \times \partial\Delta^{k-1}(0;1)\big).
\end{align}
Looking at the proof of the above fact, we see that the $\delta-$neighbourhood is obtained by keeping the $z_1-$coordinate fixed. Hence (\ref{2}) can be rewritten as
\begin{align} \label{set}
F_{n+1}^{-1}(\partial\Delta^k(0;1))\cap F(n)(\ov{\Delta^k(0;R)}) \subset N_\delta\big(\pi_1 \circ F(n) (\ov{\Delta^k(0;R)}) \times \partial\Delta^{k-1}(0;1)\big).
\end{align}
This exactly means that
\[ \Omega_{n+1,1} \cap \ov{\Delta^k(0;R)}=\{ (z_1,z_2,\hdots,z_k) \in \ov{\Delta^k(0;R)}: |z_i^{n+1}|<1 \text{ for every } 2 \le i \le k \}. \]
 
\medskip\no 
Since the automorphisms $\{F_i\}_{i=0 }^n$ satisfy condition (a),
\begin{align*}
 F(n)^{-1}\big(\mbb{C} \times \Delta^{k-1}(0;1)\big)\cap \ov{\Delta^k(0;R)}&= \Omega_{n,1} \cap \ov{\Delta^k(0;R)} 
 \end{align*}
 i.e.,
 \begin{align*}
 F(n)^{-1}\big(\mbb{C} \times \partial\Delta^{k-1}(0;1)\big)\cap \ov{\Delta^k(0;R)}&= \partial\Omega_{n,1} \cap \ov{\Delta^k(0;R)}.
 \end{align*}
 As $\ov{\Delta^k(0;R)}=F(n)^{-1} \circ F(n)\big(\ov{\Delta^k(0;R)}\big)$, the above expression further simplifies as
 \begin{align*}
 F(n)^{-1}\big(\mbb{C} \times \partial\Delta^{k-1}(0;1)\cap F(n)(\ov{\Delta^k(0;R)})\big)&= \partial \Omega_{n,1} \cap \ov{\Delta^k(0;R)}.
 \end{align*}
Now by continuity of $F(n)$, we have
\[ F(n)^{-1}\Big( N_\delta\big(\mbb{C} \times \partial\Delta^{k-1}(0;1)\big)\cap F(n)(\ov{\Delta^k(0;R)})\Big) \subset N_{\ep}\big(\partial \Omega_{n,1} \cap \ov{\Delta^k(0;R)}\big),\]
which using (\ref{set}) says that
\[ \partial\Omega_{n+1,1} \cap \ov{\Delta^k(0;R)}\subset  F(n)^{-1}\Big( N_\delta\big(\mbb{C} \times \partial\Delta^{k-1}(0;1)\big)\cap F(n)(\ov{\Delta^k(0;R)})\Big) \subset N_{\ep}\big(\partial \Omega_{n,1} \cap \ov{\Delta^k(0;R)}\big).\] This completes the proof of (i). 
%From condition (a)
%\[ \max\{|\pi_1 \circ F(n)(z)|: z \in \Omega_{n,1} \cap \ov{\Delta^k(0;R)}\}=\tilde{c}<1.\] By Lemma \ref{topological} for $\ep/3$ there exists $c_{\ep}$ such that for every $c_\ep<c<1$ 
% $$\partial\Omega_{n,c} \cap \Delta(0;R)  \in N_{\ep/3}\big(\partial \Omega_{n,1} \cap \Delta(0;R) \big).$$ Choose $c_0=\max\{ \tilde{c},c_{\ep} ,c_n\}$ and $c_1=\sqrt{\frac{c_0^2+1}{2}}>c_0.$ Then $\Omega_{n,c_1} \cap \ov{\Delta^k(0;R)}$ will satisfy
% \[ \Omega_{n,c_1} \cap \ov{\Delta^k(0;R)}=\{ (z_1,\hdots,z_k) \in \ov{\Delta^k(0;R)}: |z_i^n|<c_1 \text{ for every } 2 \le i \le k \},\] 
% $$\partial\Omega_{n,c_1} \cap \Delta(0;R)  \in N_{\ep/3}\big(\partial \Omega_{n,1} \cap \Delta(0;R) \big).$$
% 
%\medskip\no 
%Again using Lemma \ref{neighbourhood} with $1=c_1^2$ and repeating the arguments as above there exists $\alpha_{n+1}''$ such that if $F_{n+1}=F_{\alpha_{n+1}}$ where $0< \alpha_{n+1}<\alpha_{n+1}''$ then
%\[ \partial\Omega_{n+1,c_1^2} \cap \ov{\Delta^k(0;R)} \subset N_{{\ep}/{3}}\big(\partial \Omega_{n,c_1} \cap \ov{\Delta^k(0;R)}\big).\]
%Thus if $c_{n+1}=c_1^2$ and $F_{n+1}=F_{\alpha_{n+1}}$ where $$0< \alpha_{n+1}\le \min\{\alpha_{n+1}',\alpha_{n+1}'', \alpha_{n+1}(i): 0 \le i \le n\}$$ then condition (ii) is satisfied  i.e.,
%\[ \partial\Omega_{n+1,c_{n+1}} \cap \ov{\Delta^k(0;R)} \subset N_\ep\big(\partial \Omega_{n+1,1} \cap \ov{\Delta^k(0;R)}\big).\]

\medskip\no 
From Lemma \ref{topological}, there exists $0< \tilde{c}_{n+1}<1$ such that for every $\tilde{c}_{n+1} \le c<1$
\[ \partial\Omega_{n+1,c} \cap \ov{\Delta^k(0;R)} \subset N_\ep\big(\partial \Omega_{n+1,1} \cap \ov{\Delta^k(0;R)}\big).\]
Thus for $c_{n+1}=\max\{c_i,\tilde{c}_{n+1}: 0 \le i \le n\}$, property (ii) is proved.

\medskip\no 
Now as in the proof of Theorem \ref{Short C^k theorem}, note that if the $\alpha_m$'s are chosen sufficiently small, i.e., $\alpha_m \le \alpha_m(n+1)$ for every $m \ge n+2$, then there exists $c_{n+1}<c_{n+1}'< 1$ such that
\[ F(m)F(n+1)^{-1}\big(\Delta^k(0; c_{n+1})\big)\subset \Delta^k\big(0; c_{n+1}(c_{n+1}')^{m-n-1}\big).\] So for $m \ge n+2$, if  $$\alpha_m \le \min\{\alpha_{m-1}^2, \alpha_m(i): 0 \le i \le n+1\}$$ then
\[ \Omega_{n+1,c_{n+1}} \subset \Omega_{m, c_{n+1}} \subset \Omega \text{ and } \Omega_{n,c_n} \subset \Omega_{n+1,c_n} \subset \Omega_{n+1,c_{n+1}}. \] This proves property (iii).
\end{proof}
\begin{rem}\label{Remark 3}
Note that the proof of Proposition \ref{main proposition} ensures that if $0<c_{n+1}<c<1$, then 
\[\partial \Omega_{n+1,c} \cap \ov{\Delta^k(0;R)}\subset N_{\ep}\big(\partial \Omega_{n+1,1} \cap \ov{\Delta^k(0;R)}\big)\]also.
\end{rem}

\no Let us recall the following definitions from \cite{Gl2}. For a given $R \gg 1$ and for every $1 \le j \le k$, let 
\[ P_j(R)=\{ z \in \mbb C^k: |z_j|=1, |z_i|  \le R, i \neq j \}.\]

\begin{defn}
Fix $j \le k$. For a given $l\ge 1$ and $\phi \in  {C}^l(P_j(R))$, let
\[ \Gamma_{\phi}^j(R)=\{ z \in \mbb{C}^k: z_j=\phi(z_1,\hdots,z_j,\xi,z_{j+1}, \hdots,z_k)\xi, \xi \in \partial \Delta, |z_i| \le R, i \neq j \}.\]
We will refer to this as the $ {C}^l-$graph over $P_j(R)$ given by $\phi.$
\end{defn}
\begin{defn}
Fix $j \le k$. For a given $l\ge 1$ and $\phi \in  {C}^l(P_j(R))$, let 
\[ G_{\phi}^j(R)=\{ z \in \mbb{C}^k: z_j=t\phi(z_1,\hdots,z_j,\xi,z_{j+1}, \hdots,z_k)\xi, 0\le t \le 1,\xi \in \partial \Delta, |z_i| \le R, i \neq j \}.\]
This is called the standard domain over $P_j(R)$ given by $\phi.$
\end{defn}
Fix $j \le k$. For a given $\ep>0$, $l\ge 1$ and $\phi \in  {C}^l(P_j(R))$ the $\ep-$neighbourhood of $\phi$ is the collection of all $ {C}^l-$smooth functions $\psi \in P_j(R)$ such that
\[ \|\phi-\psi\|_{ {C}^l(P_j(R))} < \ep.\]

\no Now using the above results appropriately it is possible to control the boundary of a {\it Short $\mbb{C}^k$} on a large enough polydisk. This is stated as follows:
\begin{thm}\label{pre theorem 3}
For given $R \gg 1 $ and an  $0< \ep <1$, there exists a Short $\mbb C^k$, say $\Omega$ such that:
\begin{enumerate}
%\item[(i)] $\partial\Omega \cap \ov{\Delta^k(0;R)}$ is contained in $\ep-$perturbation of $\Delta(0;R) \times \partial \Delta^{k-1}(0;1)$, i.e.,  $$\Omega \subset \Delta(0;R) \times \Delta^{k-1}(0;1+\ep).$$
\item [(i)] For every $2 \le j \le k$, there exists $r_j \in  {C}^{\infty}(P_j(R))$ such that 
$\Gamma_{r_j}^j(R)$ is an $\epsilon$-small ${C}^{\infty}-$perturbation of $P_j(R).$
\item[(ii)] For every $2 \le j \le k$, $\Omega \cap \ov{\Delta^k(0;R)}$ is the intersection of standard domains over $P_j(R)$ given by $r_j$'s, i.e.,
$$\Omega \cap \ov{\Delta^k(0;R)}=\bigcap_{j=2}^k G_{r_j}^j(R).$$
\item[(iii)] $\partial\Omega \cap \ov{\Delta^k(0;R)}$ is contained in an $\ep-$perturbation of $\Delta(0;R) \times \partial \Delta^{k-1}(0;1)$, i.e.,  $$\Delta(0;R) \times \Delta^{k-1}(0;1) \subset \Omega\cap \ov{\Delta^k(0;R)} \subset \Delta(0;R) \times \Delta^{k-1}(0;1+\ep).$$
\end{enumerate}
\end{thm}

\begin{proof}
For every $2 \le j \le k$, let $r_j^{-1}$ denote the constant function $1$ on $P_j(R)$, i.e.,
\[ \Gamma_{r_j^{-1}}^j(R)=P_j(R).\]
and $\ep_n^0={\ep}/{k2^{n+1}}$ for every $n \ge 0.$

\medskip\no 
{\bf Induction statement:} For a given $n \ge 0$, there exist $(n+1)-$automorphisms of $\mbb C^k$, say $F_i$ ($0 \le i \le n$) such that:
\begin{itemize}
\item[(i)] $F_i=F_{\alpha_i}$ where $\alpha_{i+1} \le \alpha_i^2$ for $0 \le i \le n-1.$

\medskip
\item[(ii)]For every $2 \le j \le k$ and $0 \le i \le n$,
\[ S_j=\{z \in \ov{\Delta^k(0;R)}: |\pi_j \circ F(i)(z)|=1\}\] is given by the graph of $r_j^i \in  {C}^{\infty}(P_j(R))$ where
\[\|r_j^i-r_j^{i-1}\|_{ {C}^i(P_j(R))}< \ep^0_i \text{ and }\Omega_{i,1} \cap \ov{\Delta^k(0;R)}=\bigcap_{j=2}^k G_{r_j^i}^j(R).\]
\item[(iii)] For every $0 \le i \le n$ and $2 \le j \le k$, $r_j^i \in  {C}^{\infty}(P_j(R))$ is increasing, i.e., 
\[ 0<r_j^{i-1}(x)< r_j^{i}(x) \text{ for every } x \in P_j(R) \text{ and } G_{r_j^{i-1}}(R) \subset G_{r_j^{i}}(R).\]
\item[(iv)] There exist $\{c_i\}_{i=0}^n$ such that $0 < c_i< 1$ and for every $0 \le i \le n$,
\[ \partial\Omega_{i, c_i} \cap \ov{\Delta^k(0;R)} \subset N_{\ep_i^0}(\partial\Omega_{i,1}\cap \ov{\Delta^k(0;R)}).\]
\item[(v)] For every $0 \le i \le n$, there exist $(n+1)-$sequences $\{\alpha_m(i)\}_{m \ge i+1}$ such that for every $m \ge n+1$, if $F_m=F_{\alpha_m}$ where $$\alpha_m\le\min\{\alpha_m(i), \alpha_{m-1}^2: 0 \le i  \le n\}$$ and $\Omega$ is the basin of attraction of the sequence $\{F_j\}_{j=0}^{\infty}$, then
\[ \Omega_{i,c_i} \subset \Omega_{m,c_n} \subset \Omega \text{ and } \Omega_{i,c_i} \subset \Omega_{i+1, c_{i+1}} \text{ for } 0 \le i \le n-1. \]
\end{itemize}
{\it Initial case:} This corresponds to $n=0.$

\medskip\no 
Note that by Lemma \ref{neighbourhood} and \ref{topological}, there exist $\alpha_0$ and $c_0$ such that $F_0=F_{\alpha_0}$ satisfies properties (i), (ii) and (iv) above. Also, from Remark \ref{Remark 1}, $F_0$ satisfies property (iii) as well.

\medskip\no 
Now using the same arguments as in the proof of property (iii) of Proposition \ref{main proposition}, there exists a sequence $\{\alpha_m(0)\}_{m \ge 1}$ such that if $F_m=F_{\alpha_m}$ for every $m \ge 1$, where $\alpha_m\le\min\{\alpha_m(0), \alpha_{m-1}^2\}$ and $\Omega$ is the basin of attraction of the sequence $\{F_p\}_{p=0}^{\infty}$, then
\[ \Omega_{0,c_0} \subset \Omega \text{ and } \Omega_{m,c_0} \subset \Omega \]
 for every $m \ge 1.$
 
\medskip\no 
So we may assume that the above conditions are true for some $n_0 \ge 0$. 

\medskip\no 
{\it General case:} Let $R_{n_0}>0$ be such that $$F(n_0)\big(\ov{\Delta^k(0;R)}\big) \subset \subset \ov{\Delta^k(0;R_{n_0})}.$$ Recall Lemma 2.1 from \cite{Gl2}. 
\begin{lem}\label{Globevnik}
Let $l \ge 1$ and $r_0>0.$ Let $\Phi$ be a holomorphic automorphism of $\mbb C^k$ and let $R_0>1$ be so large that 
\[ \Phi(\ov{\Delta^k(0;r_0)}) \subset \subset \ov{\Delta^k(0;R_0)}.\]
Let $S=\{ z \in \ov{\Delta^k(0;R_0)}:|z_j|=1 \}$ and assume that
$ \Phi^{-1}(S) \cap \ov{\Delta^k(0;r_0)}$ is a $ {C}^l-$graph over $P_i(r_0).$ Then for a given $\ep>0$, there exists a $\delta>0$ such that if $T$ is a $ {C}^l-$graph over $P_j(R_0)$ in the $\delta-$neighbourhood of $S$, then $\Phi^{-1}(T) \cap \ov{\Delta^k(0;r_0)}$ is $ {C}^l-$smooth graph over $P_i(r_0)$ belonging to the $\ep-$neighbourhood of $ \Phi^{-1}(S) \cap \ov{\Delta^k(0;r_0)}.$
\end{lem}
\no For $2 \le j \le k$, let $S_j=\{ z \in \ov{\Delta^k(0;R_{n_0})}:|z_j|=1 \}$. By assumption, each $F(n_0)^{-1}(S_j)$ is a graph over $P_j(R).$ Then for $\ep^0_{n_0+1}$, there exists $\delta>0$ such that Lemma \ref{Globevnik} is true for the automorphism $F(n_0)$ with $l=n_0+1.$ 

\medskip\no 
Now choose $0<\ep_{n_0+1} \le \min\{\ep^0_{n_0+1}, \delta \}.$ 

\medskip\no 
By applying Lemma \ref{neighbourhood} on $\ov{\Delta^k(0;R_{n_0})}$, there exists $\alpha_{n_0+1}'$ such that if $F_{n_0+1}=F_{\alpha_{n_0+1}}$ where $0 < \alpha_{n_0+1} \le \alpha_{n_0+1}' $, then for every $2 \le j \le k$ 
\[T_j=\{z \in \ov{\Delta^k(0;R_{n_0})}: |\pi_j \circ F_{n_0+1}(z)|=1 \} \]
is in the $\ep_{n_0+1}-$neighbourhood of $S_j.$ Also from Proposition \ref{main proposition}, there exists $\alpha_{n_0+1}''$ such that if $F_{n_0+1}=F_{\alpha_{n_0+1}}$, where $0 < \alpha_{n_0+1} \le \alpha_{n_0+1}'' $, there exists $c_{n_0}\le c_{n_0+1}<1 $ such that 
\begin{align}\label{thm 3,eqn 1}
\Omega_{n_0+1} \cap \ov{\Delta^k(0;R)}=\{(z_1,z_2,\hdots,z_k) \in \ov{\Delta^k(0;R)}:|z_j^{n_0+1}|<1 \text{ for every } 2 \le j \le k\} 
\end{align} and 
\begin{align}\label{thm 3,eqn 2}
\partial \Omega_{n_0+1,c_{n_0+1}} \cap \ov{\Delta^k(0;R)}\subset N_{\ep^0_{n_0+1}}\big(\partial \Omega_{n_0+1,1} \cap \ov{\Delta^k(0;R)}\big).
\end{align}
Choose $0< \alpha_{n_0+1}\le \min\{\alpha_{n_0+1}', \alpha_{n_0+1}'',\alpha_{n_0}^2, \alpha_{n_0+1}(i): 0 \le i \le n_0\}.$ 

\medskip\no 
By assumption, $F(n_0)^{-1}(S_j) \cap \ov{\Delta^k(0;R)}$ is actually a $ {C}^{\infty}-$smooth graph over $P_j(R)$, i.e., in particular a $ {C}^{n_0+1}-$smooth graph over $P_j(R)$. Hence by Lemma \ref{Globevnik}, $F(n_0)^{-1}(T_j) \cap \ov{\Delta^k(0;R)}$ is a $ {C}^{n_0+1}-$smooth graph over $P_j(R)$. Let $r_j^{n_0+1}$ denote the function for this graph and by the choice of $\ep_{n_0+1}$, it is assured that for every $2 \le j \le k$,
\[ \|r_j^{n_0}-r_j^{n_0+1}\|_{ {C}^{n_0+1}(P_j(R))} \le \ep_{n_0+1}^0\text{ and }\Omega_{n_0+1,1} \cap \ov{\Delta^k(0;R)}=\bigcap_{j=2}^k G_{r_j^{n_0+1}}^j(R).\]
Observe that for every $2 \le j \le k$, by  Remark \ref{Remark 1}, $T_j$ is a smooth graph in $P_j(R_{n_0}).$ Also $F(n_0)^{-1}(T_j) \cap \ov{\Delta^k(0;R)}$ is a graph over $P_j(R)$ and $F(n_0)^{-1}$ is a smooth function. Hence $r_j^{n_0+1} \in  {C}^{\infty}(P_j(R)).$

\medskip\no 
Note that by construction $\Omega_{n_0+1}$ satisfies condition (\ref{thm 3,eqn 1}) and (\ref{thm 3,eqn 2}), for $c_{n_0+1}.$ So $\{F_i\}_{i=0}^{n_0+1}$ satisfies properties (i), (ii) and (iv) of the induction statement.

\medskip\no 
From Remark \ref{Remark 1}, it follows that  for every $2 \le j \le k$, 
\[ \{z \in \ov{\Delta^k(0;R_{n_0})}: |z_j| \le 1 \} \subset \{z \in \ov{\Delta^k(0;R_{n_0})}: |\pi_j \circ F_{n_0+1}(z)| \le 1\}.\]
Hence property (iii) is also satisfied, i.e., for every $0 \le i \le n_0+1$
$$r_{j}^{i-1}<r_j^{i} \text{ and } G_{r_j^{i-1}}^j \subset G_{r_j^{i}}^j.$$

\medskip\no 
Now as observed in the proof of Proposition \ref{main proposition}, the sequence $\{\alpha_m(n_0+1)\}_{m \ge n_0+2}$ can be appropriately chosen such that if $F_m=F_{\alpha_m}$ for every $m \ge n_0+2$, where $$\alpha_m\le\min\{\alpha_m(i), \alpha_{m-1}^2: 0 \le i \le n_0+1\}$$ and $\Omega$ is the basin of attraction of the sequence $\{F_j\}_{j=0}^{\infty}$ at the origin, then
\[ \Omega_{i,c_i} \subset \Omega_{i+1, c_{i+1}}  \subset \Omega_{m,c_{n_0+1}} \subset \Omega  \text{ for } 0 \le i \le n_0. \]Thus the induction statement is true for every $n \ge 0.$

\medskip\no 
So we have a sequence of automorphisms $\{F_p\}_{p \ge 0}$ such that the following is true for every $p \ge 0$:
\begin{itemize}
\item The basin of attraction of $\{F_p\}_{p \ge 0}$ at the origin, i.e., $\Omega$ is a {\it Short $\mbb{C}^k$}.
\item There exists an increasing sequence $\{c_p\}_{p\ge 0}$ such that $0 <c_p<1$ and $$\Omega_{p,c_p} \subset \Omega_{p+1, c_{p+1}} \subset \Omega.$$
\item By Remark \ref{Remark 3}, $$ \partial\Omega_{p, c} \cap \ov{\Delta^k(0;R)} \subset N_{\ep_p^0}(\partial\Omega_{p,1}\cap \ov{\Delta^k(0;R)})$$ for every $p \ge 0$, if $0<c_p \le c<1$. 
\item For every $2 \le j \le k $, there exists an increasing sequence of positive functions $r_j^i \in  {C}^{\infty}(P_j(R))$ such that $$\Omega_{p,1} \cap \ov{\Delta^k(0;R)}=\bigcap_{j=2}^k G_{r_j^{p}}^j(R),\;\; \Omega_{p,1} \cap \ov{\Delta^k(0;R)} \subset \Omega_{p+1,1} \cap \ov{\Delta^k(0;R)}$$ and   $$ \|r_j^{p-1}-r_j^{p}\|_{ {C}^p(P_j(R))} \le \ep_{p}^0.$$
\end{itemize}
{\it Case 1:} Suppose $0<c_n<C<1$ for every $n \ge 0.$ Then from Theorem \ref{Short C^k theorem}, there exists $n_0 \ge 0$ such that 
$$\Omega=\bigcup_{n \ge n_0} \Omega_{n,C} \text{ and } \Omega_{n,C} \subset \Omega_{n+1,C} \text{ for }n \ge n_0.$$
So 
\[ \partial\Omega \cap \ov{\Delta^k(0;R)}=\lim_{n \to \infty} \partial \Omega_{n,C} \cap \ov{\Delta^k(0;R)}.\]
But $$ \partial\Omega_{n, C} \cap \ov{\Delta^k(0;R)} \subset N_{\ep_n^0}(\partial\Omega_{n,1}\cap \ov{\Delta^k(0;R)})$$ and $\ep_n^0 \to 0$ as $n \to \infty.$ Hence
\[ \Omega \cap \ov{\Delta^k(0;R)}=\lim_{n \to \infty} \Omega_{n,C} \cap \ov{\Delta^k(0;R)}=\lim_{n \to \infty} \Omega_{n,1} \cap \ov{\Delta^k(0;R)}.\]

\medskip\no 
{\it Case 2:} Suppose $c_n \to 1$ as $n \to \infty.$ In this case, for every $n \ge 0$
\begin{align*}
 \Omega_{n,c_n} \subset \Omega, \text{ and } \partial\Omega_{n,c_n} \cap \ov{\Delta^k(0;R)} \subset N_{\ep_n^0}\big(\partial\Omega_{n,1} \cap \ov{\Delta^k(0;R)}\big).
 \end{align*}
%\[ \Omega=\bigcup_{n=0}^{\infty} \Omega_{n,c} \text{ for some } 0<c<1. \]
Notice that by choice $c_0 \le c_n < 1$ and $\Omega_{n,c_0} \subseteq \Omega_{n,c_n}$. Hence
\[ \Omega = \bigcup_{n=n_0}^{\infty} \Omega_{n,c_0}\subset \bigcup_{n=0}^{\infty} \Omega_{n,c_n}\subseteq \Omega,\] i.e.,
\[ \Omega=\bigcup_{n=0}^{\infty} \Omega_{n,c_n} \text{ and } \Omega_{n,c_n} \subset \Omega_{n+1,c_{n+1}}.\]
So
\[ \partial\Omega \cap \ov{\Delta^k(0;R)} =\lim_{n \to \infty} \partial \Omega_{n,c_n} \cap \ov{\Delta^k(0;R)} \]
and $\ep_n^0 \to 0$ as $n \to \infty.$ Hence
\[ \Omega \cap \ov{\Delta^k(0;R)}=\lim_{n \to \infty}\Omega_{n,c_n} \cap \ov{\Delta^k(0;R)}=\lim_{n \to \infty}  \Omega_{n,1} \cap \ov{\Delta^k(0;R)}.\]
Now  for every $2 \le j \le k$ and for a fixed $i \ge 0$, the sequence of functions $\{r_j^n\}_{n=i}^{\infty}$ is an increasing Cauchy sequence in $ {C}^{i}(P_j(R))$. Since $ {C}^{i}(P_j(R))$ is a Banach space and $r_j^n \to r_j$ as $n \to \infty$, $r_j \in  {C}^i(P_j(R))$ for every $i \ge 0$, i.e., $r_j \in  {C}^{\infty}(P_j(R))$ for every $2 \le j\le k.$ Also 
\[ \|r_j-1\|_{ {C}^l(P_j(R))} \le \ep_0 ,\]i.e.,\[ 1-\ep_0  \le r_j(x)\le 1+\ep_0\] for every $ x \in P_j(R)$ and
$$\Omega_{n,1} \cap \ov{\Delta^k(0;R)}=\bigcap_{j=2}^k G_{r_j^{n}}^j(R).$$ Hence taking limits as $n \to \infty$ on both sides,
$$\Omega \cap \ov{\Delta^k(0;R)}=\bigcap_{j=2}^k G_{r_j}^j(R).$$ 
But $G_{r_j}^j$ is contained in the $\ep_0-$neighbourhood of $P_j(R)$, i.e., in $N_{\ep_0}(P_j(R))$, so we have that
\[\Omega \cap \ov{\Delta^k(0;R)} \subset \ov{\Delta(0;R)} \times \Delta^{k-1}(0;1+\ep).\] 
Note that by construction, for every $n \ge 0$,
\[ \ov{\Delta(0;R)} \times \Delta^{k-1}(0;1) \subset \bigcap_{j=2}^k G_{r_j^n}^j(R), \]
i.e.,
\[ \ov{\Delta(0;R)} \times \Delta^{k-1}(0;1) \subset \Omega \cap \ov{\Delta^k(0;R)}\] and this completes the proof.
\end{proof}
\begin{cor}
There exists a {\it Short $\mbb{C}^k$}, say $\Omega$ and a bounded domain $D \subset \mbb{C}^{k-1}$ which embeds holomorphically in $\Omega$.
\end{cor}
\begin{proof}
Consider the subspace of $\mbb{C}^k$ given by
\[H=\{ (z_1,z_2,\hdots,z_k) \in \mbb{C}^k: z_1=0 \}.\]
Let $\Omega$ be the {\it Short $\mbb{C}^k$} obtained from Theorem \ref{pre theorem 3} and let
$D$ be the component of $\Omega \cap H$ which contains the origin. From Theorem \ref{pre theorem 3}, it follows that $D$ is an $\ep-$perturbation of the unit polydisk $\Delta^{k-1}(0;1)$ where $\ep>0$ is sufficiently small.
\end{proof}
\section{Proof of Theorem \ref{theorem 3} and Theorem \ref{theorem 4}}
\no The main idea here is to use Theorem \ref{pre theorem 3} for polydisks of increasing radii in $\mbb{C}^k$, $k \ge 2.$ We will follow the same notations  from Section 4. First we prove Theorem \ref{theorem 3}, i.e., when $k=2.$

\begin{proof}[Proof of Theorem \ref{theorem 3}]
Choose $0<\ep<10$ and a sequence of strictly increasing positive real numbers $\{R_n\}_{n \ge 0}$ such that $R_0 \gg 1.$ 
Let $r$ denote the constant function $1$ on all of $\mbb C^2$, then for every $R\ge 0$,
\[ \Gamma_{r}^2(R)=P_2(R).\]
Let $\ep_n={\ep}/{2^{n+2}}$ for every $n \ge 0.$ %Also for every $n \ge 0$ let $F(n,n-1)=F_n$

\medskip\no 
{\bf Induction statement:} For a given $n \ge 0$, there exist $(n+1)-$automorphisms of $\mbb C^2$, say $F_i$ ($0 \le i \le n$) such that
\begin{itemize}
\item[(i)] $F_i=F_{\alpha_i}$, where $\alpha_{i+1} \le \alpha_i^2$ for $0 \le i \le n-1.$

\medskip
\item[(ii)]There exist $(n+1)-$increasing positive real numbers $\{R_i'\}_{i=1}^{n+1}$ such that  
\[ F(i)\big(\ov{\Delta^2(0;R_{i+1})}\big) \subset\subset \ov{\Delta^2(0;R_{i+1}')} \text{ and }F_i\big(\ov{\Delta^2(0;R_i')}\big) \subset \subset \ov{\Delta^2(0;R_{i+1}')}\]
for every $0 \le i \le n$. Here $R_0'=R_1.$ Also
\[F_i^{-1}\big(\partial\Delta^2(0;1)\big) \cap \ov{\Delta^2(0;R_i')} \subset N_{\ep_i}\big( \ov{\Delta(0;R_i')} \times \partial \Delta(0;1)\big).\]

\medskip
\item[(iii)] For every $0 \le i \le n$ and $i \le j \le n$ there exist functions $r_i^j \in  {C}^{\infty}\big(P_2(R_{i+1}')\big)$ whose graphs over $P_2(R_{i+1}')$ are the $(n-i+1)$ sets
\[S(i,j)=\{z \in \ov{\Delta^2(0;R_{i+1}')}: |\pi_2 \circ F(i,j)(z)|=1\}.\]

\medskip
\item[(iv)]For every $0 \le i \le n$ and $i\le j \le n$,
\[ \|r_i^{j+1}-r_i^j\|_{ {C}^i\big(P_2(R_{i+1}')\big)}< {\ep_i}/{2^{j+1-i}} \] and  \[G_{r_i^{j}} ^2(R_{i+1}') \subset G_{r_i^{j+1}} ^2(R_{i+1}').\]
Moreover $r_i^i=r$, i.e., $G_{r_i^i} ^2(R_{i+1}')=P_2(R_{i+1}').$

\medskip
\item[(v)] There exist $\{c_i\}_{i=0}^n$ such that $0 < c_i< 1$ and for every $0 \le i \le n$ and $i \le j \le n$,
\[ \partial\Omega_{j, c_j} \cap F(i)^{-1}\big(\ov{\Delta^2(0;R_{i+1}')}\big)\subset N_{\ep_j}\Big(\partial\Omega_{j,1}\cap F(i)^{-1}\big(\ov{\Delta^2(0;R_{i+1}')}\big)\Big).\]
\item[(vi)] For every $0 \le i \le n$, there exist $(n+1)-$sequences $\{\alpha_m(i)\}_{m \ge i+1}$ such that $$\alpha_m\le\min\{\alpha_m(i), \alpha_{m-1}^2: 0 \le i  \le n\}$$  for $m \ge n+1$ and corresponding automorphisms $F_m=F_{\alpha_m}$ with the property that if $\Omega$ is the basin of attraction of the sequence $\{F_j\}_{j\ge 0}$ at the origin, then
\[\Omega_{i,c_i} \subset \Omega_{i+1, c_{i+1}} \subset \Omega_{m,c_n} \subset \Omega   \text{ for } 0 \le i \le n-1. \]
\end{itemize}
{\it Initial case:} Suppose $n=0.$

\medskip\no 
By Lemma \ref{neighbourhood}, for $R=R_0'>R_0$, $\ep=\ep_0$ and $l=0$, there exists $0<\alpha_0<1$ such that $F_0=F_{\alpha_0}$ satisfies (i) and (ii) above. Let $R_1'>0$ be such that
\[ F_0\big(\ov{(\Delta^2(0;R_1)}\big) \subset\subset \ov{\Delta^2(0;R_1')} .\] But $S(0,0)=P_2(R_1')$ and hence (iii) holds by definition. Also, (iv) is vacuous, since $r^1_0$ is not yet defined. By Lemma \ref{topological}, if $K=F_0^{-1}\big(\ov{\Delta^2(0;R_1')}\big)$, there exists $0<c_0<1$ such that (v) holds. 

\medskip\no 
Finally, as in Theorem \ref{pre theorem 3}, the same arguments as in the proof of property (iii) of Proposition \ref{main proposition}, there exists a sequence $\{\alpha_m(0)\}_{m \ge 1}$ such that if $F_m=F_{\alpha_m}$ for every $m \ge 1$, where $\alpha_m\le\min\{\alpha_m(0), \alpha_{m-1}^2\}$ and $\Omega$ is the basin of attraction of the sequence $\{F_p\}_{p=0}^{\infty}$ at the origin, then
\[ \Omega_{0,c_0} \subset \Omega \text{ and } \Omega_{m,c_0} \subset \Omega\]  for every  $m \ge 1.$ Hence $F_0$ satisfies all of (i)--(vi).

\medskip\no 
By induction we may assume that the above properties are true for some $n_0 \ge 0.$

\medskip\no 
{\it General Case:} By the induction hypothesis, there exists $R'_{n_0+1}$ such that
\[F(n_0)\big(\ov{\Delta^2(0;R_{n_0+1})}\big) \subset\subset \ov{\Delta^2(0;R_{n_0+1}')} \text{ and }F_{n_0}\big(\ov{\Delta^2(0;R_{n_0}')}\big) \subset \subset \ov{\Delta^2(0;R_{n_0+1}')}.\]
By Lemma \ref{neighbourhood}, for $R=R_{n_0+1}'>R_0$, $\ep=\ep_{n_0+1}$ and $l=n_0$, there exists $0<\alpha'<1$ such that if $F_{n_0+1}=F_{\alpha}$ where $0<\alpha<\alpha'$, 
the above properties (i) and (ii) are true. Pick $R_{n_0+2}'>0$ such that
\[ \alpha' R'_{n_0+1}+{R'}^2_{n_0+1}<R_{n_0+2}'.\]
Also for every $0 \le i \le n_0$, pick $\tilde{R}_i>0$ such that
\[ F(i,n_0)\big(\ov{\Delta^2(0;R'_{i+1})}\big) \subset \subset\ov{ \Delta^2(0;\tilde{R}_i)}.\]
Note that $\tilde{S}_i$ by definition is the graph  $P_2(\tilde{R}_i).$ Further, from the induction hypothesis, \\ ${F(i,n_0)}^{-1}(\tilde{S}_i)$ is a graph over $\ov{\Delta^2(0;R'_{i+1})}$ for every $0 \le i \le n_0.$

\medskip\no 
For a fixed $i$, $0 \le i \le n_0$, by an application of Lemma \ref{Globevnik} on $\Delta^2(0;\tilde{R}_i)$ for $l=i$, there exists a $\delta_i>0$ such that if $\tilde{T}_i$ is a graph over $P_2(\tilde{R}_i)$ which lies in the $\delta_i-$neighbourhood of $\tilde{S}_i$, then $F(i,n_0)^{-1}(\tilde{T}_i)$ is a graph on $\Delta^2(0;R'_{i+1})$. Moreover, $F(i,n_0)^{-1}(\tilde{T}_i)$ lies in the ${\ep_i}/{2^{n_0+1-i}}-$neighbourhood of $F(i,n_0)^{-1}(\tilde{S}_i).$
Now applying Lemma \ref{neighbourhood} repeatedly on each $\Delta^2(0;\tilde{R}_i)$, there exists $\alpha'_i$ such that for $0<\alpha \le \min\{\alpha'_i: 0 \le i \le n_0\} $, the set
\[|\pi_2 \circ F_{\alpha}(z)|=1 \text{ for } z \in \Delta^2(0; \tilde{R}_i) \]
is the graph of a smooth function $\phi_{\alpha}^i$ on $P_2(\tilde{R}_i).$ Also, for every $0 \le i \le n_0$
 \[ \|\phi_{\alpha}^i -1\|_{ {C}^i(P_2(\tilde{R}_i))}< \delta_i.\]
 Let $F_{n_0+1}=F_{\alpha_{n_0+1}}$ where $$0<\alpha_{n_0+1}<\min\{\alpha',\alpha_i',\alpha_i(n_0+1),\alpha_{n_0}^2: 0 \le i \le n_0 \} $$
Hence from Lemma \ref{Globevnik}, there exists $r_i^{n_0+1} \in  {C}^{\infty}(P_2(R'_{i+1}))$ such that 
\[ \Gamma^2_{r_i^{n_0+1}}(R'_{i+1})=F(i,n_0)^{-1}\big(\Gamma^2_{\phi_{\alpha_{n_0+1}}^i}(\tilde{R}_i)\big)\cap \Delta^2(0;R'_{i+1})\]
and
\[\|r_i^{n_0+1}-r_i^{n_0}\|_{ {C}^i(P_2(R'_{i+1}))} \le {\ep_i}/{2^{n_0+1-i}}.\] Moreover, by construction \[S(i,n_0+1)=\Gamma^2_{r_i^{n_0+1}}(R'_{i+1})\]  for $0 \le i \le n_0$  and $S(n_0+1,n_0+1)=P_2(R'_{n_0+2})$ which is also a graph. From Remark \ref{Remark 1} and the fact that $\alpha_{n_0+1}< \alpha'$, the sequence $\{F_i\}_{i=0}^{n_0+1}$ satisfies all the properties (i)--(iv).

\medskip\no 
For properties (v)--(vi) we will use the same arguments as in the proof of the initial case. Let $$K=F(n_0+1)^{-1}\big(\ov{\Delta^2(0;R'_{n_0+2})}\big).$$ Then by Theorem \ref{topological}, there exists $0<c_{n_0}\le c_{n_0+1}<1$ such that (v) holds. Finally, by imitating the argument in the proof of Proposition \ref{main proposition}(iii) there exists a sequence $\{\alpha_m(n_0+1)\}_{m \ge n_0+2}$ such that for $m \ge n_0+2$ if $F_m=F_{\alpha_m}$ where
$$0<\alpha_m \le\min\{\alpha_m(i), \alpha_{m-1}^2:0 \le i \le n_0+1\}$$ and $\Omega$ is the basin of attraction of $\{F_p\}_{p=0}^{\infty}$ at the origin, then 
\[ \Omega_{i,c_i} \subset \Omega_{i+1, c_{i+1}} \subset \Omega_{m,c_{n_0+1}} \subset \Omega\]   for $0 \le i \le n_0. $
To conclude, the induction statement is true for every $n \ge 0.$ Therefore
we have a sequence of automorphisms $\{F_p\}_{p \ge 0}$  such that the following is true for every $p \ge 0$.
\begin{itemize}
\item The basin of attraction of $\{F_p\}_{p \ge 0}$ at the origin, i.e., $\Omega$ is a {\it Short $\mbb{C}^k$}.
\item There exists an increasing sequence $\{c_p\}_{p\ge 0}$ such that $0 <c_p<1$ and $$\Omega_{p,c_p} \subset \Omega_{p+1, c_{p+1}} \subset \Omega.$$
\item There exists an increasing sequence of compact sets, i.e, $K_p=F(p)^{-1}(\ov{\Delta^2(0;R_{p+1}')})$ such that
$$ \mbb{C}^2= \bigcup_{p \ge 0} \ov{\Delta^2(0; R_{p+1})} \subset \bigcup_{p \ge 0} K_p=\mbb{C}^2.$$
\item By Remark \ref{Remark 3}, $$ \partial\Omega_{p, c} \cap K_p \subset N_{\ep_p}(\partial\Omega_{p,1}\cap K_p)$$ for every $p \ge 0$ if $0<c_p \le c<1.$ 
\item For every $p \ge 0$, there exists an increasing sequence of positive functions $\{r_p^{p+i}\}_{i \ge 0} \subset  {C}^{\infty}(P_2(R'_{p+1}))$ such that \[F(p)\big(\Omega_{p+i,1}\big) \cap \ov{\Delta^2(0;R'_{p+1})}= G_{r_p^{p+i}}^2(R'_{p+1}),\]\[ \Omega_{p,1} \cap K_p \subset \Omega_{p+1,1} \cap K_p\] and   $$ \|r_p^{p+i+1}-r_p^{p+i}\|_{ {C}^p(P_2(R'_{p+1}))} \le \ep_{p}/{2^{i+1}}.$$
\end{itemize}
To summarize  
\begin{align}\label{theorem 3, eqn 1}
F(p)\big(\partial \Omega_{p+i,1}\big) \cap F(p)(K_p)= \Gamma^2_{r_p^{p+i}} (R_{p+1}')
\end{align}
for a fixed $p \ge 0 $ and for every $i \ge 0 $. Here $\{r^{p+i}_p\}_{i \ge 0}$ is  a Cauchy sequence in $ {C}^p(P_2(R'_{p+1}))$. Since $ {C}^p(P_2(R'_{p+1}))$ is complete, $r_{p}^{p+i} \to r_p$ on $P_2(R'_{p+1}).$ Now following the same arguments as in the proof of Theorem \ref{pre theorem 3},
\begin{align*}
\lim_{i \to \infty} \partial \Omega_{p+i,1} \cap K_p&=\partial \Omega \cap K_p, \\
\text{i.e., }\lim_{i \to \infty} F(p) \big(\partial \Omega_{p+i,1}\big) \cap \ov{\Delta^2(0;R'_{p+1})}&=F(p)\big(\partial \Omega\big) \cap \ov{\Delta^2(0;R'_{p+1})}
\end{align*}
 for a fixed $p$.
Also from (\ref{theorem 3, eqn 1}), it follows that
$$F(p)\big(\partial \Omega\big) \cap \ov{\Delta^2(0;R'_{p+1})}=\Gamma^2_{r_p} (R_{p+1}').$$
Hence $F(p)^{-1}\big(\Gamma^2_{r_p} (R_{p+1}')\big)$ is a $ {C}^p-$smooth hypersurface on $K_p.$ But note that the above arguments are true for every $p \ge 0$ and $K_p$ is an exhaustion of $\mbb{C}^2$. Hence, the boundary of $\Omega$ is $ {C}^{\infty}-$smooth in $\mbb{C}^2.$
\end{proof}
\no 
By extending these arguments to $\mbb{C}^k$, $k \ge 3$ in exactly the same manner as in the proof of Theorem \ref{pre theorem 3}, it is possible to obtain a {\it Short $\mbb{C}^k$} whose boundary lies in the intersection of  smooth hypersurfaces.

\medskip\no 
However, to complete the proof of Theorem \ref{theorem 4}, we will need to do a little more work. Recall the definition of a piecewise smooth pseudoconvex domain in $\mbb{C}^k$, $k \ge 2.$

\medskip\no 
\begin{defn}
Let $n \ge 1$ and $D_i$ for $1 \le i \le n$ be pseudoconvex domains in $\mbb C^k$ with $ {C}^{\infty}-$boundary in $\mbb{C}^k.$ 
Let $\rho_i$ be a $ {C}^{\infty}-$defining function for $D_i,$  i.e., $$D_i=\{z \in \mbb{C}^k: \rho_i(z)<0\}$$ such that $\nabla \rho_i \neq 0$ on $\partial \rho_i$. Then $$D= \cap_{i=1}^n D_i$$ is a piecewise smooth pseudoconvex domain if
\[ d \rho_{i_1} \wedge d \rho_{i_2} \wedge \hdots \wedge d \rho_{i_l} \neq 0\]
on $\{ z \in \partial D: \rho_{i_1}=\rho_{i_2}=\cdots=\rho_{i_l}=0\}$ for every $1 \le i_1 < i_2 <\hdots<i_l\le n.$
\end{defn}
\medskip\no 
\begin{proof}[Proof of Theorem \ref{theorem 4}]
The proof is divided into three steps.

\medskip\no 
{\it Step 1:} For a given algebraic variety of codimension 2 in $\mbb{C}^k$, $k \ge 3$, Theorem \ref{pre theorem 1} shows that for a given $\ep>0$ and for every $0< \ep < R^{-1}$, 
\begin{itemize}
\item There exists an appropriate change of coordinates, say $\phi_{\ep}$ such that
\[ \phi_{\ep}(\widetilde{V}) \subset V_R^+.\]
\item There exists a sequence of automorphisms of the form $\{F_{\alpha_p}\}_{p \ge 0}$ such that the non--autonomous basin of attraction at the origin, say $\Omega$ is a {\it Short $\mbb{C}^k$}. Further,
\[ \Omega \subset V_R \cup V_R^- \text{ and } \Omega \cap \phi_{\ep}(\widetilde{V})=\emptyset.\] 
\end{itemize}
Here $V_R$, $V_R^-$ and $V_R^+$ are as defined in Section 2. Hence to complete the proof we only need to construct an appropriate sequence of automorphisms of the form $\{F_{\alpha_p}\}_{p \ge 0}.$

\medskip\no 
{\it Step 2:} We follow the arguments in Theorem \ref{theorem 3} in $\mbb{C}^k$, $k \ge 3$ with $(k-1)$ intersections of graphs at each stage. The preliminary step will be the same as in Theorem \ref{pre theorem 3}. Then it is possible to construct a sequence of automorphisms $\{F_{\alpha_p}\}_{p \ge 0}$ and an exhaustion of $\mbb{C}^k$, i.e., $\mbb{C}^k=\cup_{p \ge 0}K_p$ such that each $$K_p=F(p)^{-1}\big(\Delta^k(0;R_{p+1}')\big) \text{ and } F(p)\big(\partial \Omega\big) \cap \Delta^k(0;R_{p+1}')=\bigcap_{i=2}^k \Gamma^i_{r_p^i}(R_{p+1}'),$$
where $\Omega$ is the basin of attraction of $\{F_p\}_{p \ge 0}$ at the origin, $\{r_p^i\}_{p \ge 0} \in  {C}^p(P_i(R_{p+1}'))$ for every $2 \le i \le k$ and $\{R_p'\}_{p \ge 1}$ is a strictly increasing sequence of positive real numbers. But now this observation is true for every $p \ge 0$ and $K_p$ is increasing i.e., $\{r_p^i\}_{p \ge 0} \in  {C}^\infty(P_i(R_{p+1}'))$. Hence the domain $\Omega$ is a {\it Short $\mbb{C}^k$} whose boundary lies in the intersection of smooth hypersurfaces. 

\medskip\no 
{\it Step 3:} Recall that $G^i_{r_p^i}(R_{p+1}')$ is the standard domain of $r_p^i$ over $P_i(R_{p+1}')$. Let 
\[ G^i= \bigcup_{p = 0}^{\infty} F(p)^{-1}\big(G^i_{r_p^i}(R_{p+1}')\big)\] for each $2 \le i \le k.$ 

\medskip\no 
{\it Claim:} $G^i$ is a smooth pseudoconvex domain for every $2 \le i \le k.$

\medskip\no 
Note that for every $2 \le i \le k$ and $p \ge 0$
\[ F(p)\big(G^i \cap K_p\big)=\{z \in \Delta^k(0;R_{p+1}'): \rho_p^i(z)<0\}\]
where $\rho_p^i(z)=|z_i|-r_p^i(z_1,\hdots,\xi_i,\hdots,z_k)$ and $\xi_i=e^{i \text{Arg}z_i}$. So locally $\rho_p^i$ is a smooth defining function.

\medskip\no 
By construction, for a fixed $2 \le i \le k$ and for a given arbitrarily small $\ep>0$, we have that
\[ \|r_p^i-1\|_{ {C}^p(P_i(R_{p+1}'))}\le \ep_p={\ep}/{2^{p}}.\]
{\it Case 1:} For $p \ge 2$ when $i \neq j,k$
\begin{align}\label{partial-1}
 \Big\|\frac{\partial r_p^i}{\partial z_j}\Big\|,\; \Big\|\frac{\partial r_p^i}{\partial \bar{z}_j}\Big\| \le \ep_p , \text{ and }\Big\|\frac{\partial^2 r_p^i}{\partial \bar{z}_k\partial z_j}\Big\|,\Big\|\frac{\partial^2 r_p^i}{\partial \bar{z}_j\partial z_k}\Big\|\le \ep_p,
 \end{align}
i.e.,
\begin{align}\label{partial-2} \Big\|\frac{\partial \rho_p^i}{\partial z_j}\Big\|,\; \Big\|\frac{\partial \rho_p^i}{\partial \bar{z}_j}\Big\| \le \ep_p , \text{ and }\Big\|\frac{\partial^2 \rho_p^i}{\partial \bar{z}_k\partial z_j}\Big\|,\Big\|\frac{\partial^2 \rho_p^i}{\partial \bar{z}_j\partial z_k}\Big\|\le \ep_p.
\end{align}
{\it Case 2:} For $p \ge 2$
\begin{align}\label{partial-3}\Big\|\frac{\partial r_p^i}{\partial \theta_i}\Big\| \le \ep_p, \text{ i.e., } \Big\|\frac{\partial r_p^i}{\partial z_i}\Big\|, \;\Big\|\frac{\partial r_p^i}{\partial \bar{z}_i}\Big\| \le \tilde{C}\ep_p
\end{align}
for some $\tilde{C}>0$. When $i \neq j$ by similar computations we see that
\begin{align}\label{partial-4}\Big\|\frac{\partial^2 \rho_p^i}{\partial \bar{z}_j\partial z_i}\Big\|,\Big\|\frac{\partial^2 \rho_p^i}{\partial \bar{z}_i\partial z_j}\Big\|\le \tilde{C}\ep_p \text{ and }\Big\|\frac{\partial^2 r_p^i}{\partial \bar{z}_i\partial z_i}\Big\|\le\tilde{C} \ep_p,
\end{align}
i.e., there exists $C>0$ such that
\begin{align}\label{partial-5}\frac{\partial^2 \rho_p^i}{\partial \bar{z}_i\partial z_i}>C-\tilde{C} \ep_p \text{ and } \Big\|\frac{\partial \rho_p^i}{\partial z_i}\Big\|, \; \Big\|\frac{\partial \rho_p^i}{\partial \bar{z}_i}\Big\| > C-\tilde{C}\ep_p.
\end{align}
Since $\ep_p \to 0$ as $p \to \infty$, it follows that if $F(p)^{-1}(w) \in \partial G_i \cap K_p$, then 
\[\sum_{m,n=1}^{k}\frac{\partial^2 \rho_p^i}{\partial \bar{z}_m\partial z_n}(w)t_m t_n >0 \text{ whenever } \sum_{n=1}^k\frac{\partial \rho_p^i}{\partial {z}_n}(w)t_n=0\]from (\ref{partial-1})-- (\ref{partial-5}).
Since $K_p$ is an exhaustion of $\mbb{C}^k$, $G_i$ is a pseudoconvex domain. 

\medskip\no 
Observe that the basin of attraction of $\{F_{\alpha_p}\}_{p \ge 0}$ is the intersection of the domains $G^i$ where $2 \le i \le k$, i.e.,
\[ \Omega=\bigcap_{i=2}^k G^i.\]
Lastly, by the same arguments as above and from (\ref{partial-1})--(\ref{partial-5})
\[ d \rho_{p}^{i_1} \wedge d \rho_{p}^{i_2} \wedge \hdots \wedge d \rho_{p}^{i_l} \neq 0\]
on $\{ w \in F(p)\big(\partial \Omega \cap K_p\big): \rho_{p}^{i_1}=\rho_{p}^{i_2}=\cdots=\rho_{p}^{i_l}=0\}$ for every $2 \le i_1 < i_2 <\hdots<i_l\le k$ whenever $p \ge 0$. This completes the proof.
\end{proof}
	\bibliographystyle{amsplain}
	\bibliography{ref}
	\end{document}